\documentclass[11pt]{article}
\usepackage{amsmath,amssymb,amsthm,amsfonts,amstext,amsbsy,amscd,color}
\usepackage{graphicx}

\def\<{\langle}
\def\>{\rangle}

\def\e{\eps}

\def\Chi{\raise .3ex
\hbox{\large $\chi$}} 
\newcommand{\norme}[1]{ {\big\lVert  #1\big\rVert}}

\newcommand{\ve}{\varepsilon}
\def\({\Bigl (}
\def\){\Bigr )}
\newcommand{\be}{\begin{equation}}
\newcommand{\p}{\partial}

\newcommand{\f}{\frac}
\newcommand{\ee}{\end{equation}}
\newcommand{\bea}{$$ \begin{array}{lll}}
\newcommand{\eea}{\end{array} $$}
\newcommand{\bi}{\begin{itemize}}
\newcommand{\ei}{\end{itemize}}

\numberwithin{equation}{section}

\newtheorem{prop}{Proposition}
\newtheorem{theorem}{Theorem}

\newtheorem{lemma}{Lemma}
\newtheorem{assumption}{Assumption}

\newtheorem{remark}{Remark}

\DeclareMathOperator{\E}{{\mathbb E}}

\DeclareMathOperator{\R}{{\mathbb R}}

\DeclareMathOperator{\PP}{{\mathbb P}}

\DeclareMathOperator{\argmin}{argmin}
\DeclareMathOperator{\HH}{{\mathcal H}}

\DeclareMathOperator{\Var}{Var} 

\providecommand{\eps}{\varepsilon}
\renewcommand{\cdot}{{\scriptstyle \bullet} }
\renewenvironment{proof}{\noindent{\bf Proof.}}{\hfill
  $\blacksquare$\par\noindent}

\begin{document}
\title{Nonparametric estimation of the division rate of a size-structured population}
\author{M. Doumic\footnote{INRIA Rocquencourt, projet BANG, Domaine de Volceau, BP 105, 781153 Rocquencourt, France. {\it email}: marie.doumic-jauffret@inria.fr}, M. Hoffmann\footnote{ENSAE-CREST and CNRS-UMR 8050, 3, avenue Pierre Larousse, 92245 Malakoff Cedex, France. {\it email}: marc.hoffmann@ensae.fr}, P. Reynaud-Bouret\footnote{CNRS-UMR 6621 and Universit\'e de Nice Sophia-Antipolis, Laboratoire J-A Dieudonn\'e, Parc
Valrose, 06108 Nice
cedex 02, France. {\it email}: patricia.reynaud-bouret@unice.fr} \;and V. Rivoirard\footnote{CEREMADE, CNRS-UMR 7534, Universit\'e Paris Dauphine, Place Mar\'echal de Lattre de Tassigny, 75775 Paris Cedex 16, France. INRIA Paris-Rocquencourt, projet Classic. {\it email}: Vincent.Rivoirard@dauphine.fr}}

\maketitle

\begin{abstract}
We consider the problem of estimating  the division rate of a size-structured population in a nonparametric setting. The size of the system evolves according to a transport-fragmentation equation: each individual grows with a given transport rate, and splits into two offsprings of the same size, following a binary fragmentation process with unknown division rate that depends on its size. In contrast to a deterministic inverse problem approach, as in \cite{PZ, DPZ}, we take in this paper the perspective of statistical inference: our data consists in a large sample of the size of individuals, when the evolution of the system is close to its time-asymptotic behavior, so that it can be related to the eigenproblem of the considered transport-fragmentation equation (see \cite{PR} for instance).  By estimating statistically each term of the eigenvalue problem and by suitably inverting a certain linear operator (see \cite{DPZ}), we are able to construct a more realistic estimator of the division rate that achieves the same optimal error bound as in related deterministic inverse problems. Our procedure relies on kernel methods with automatic bandwidth selection. It is inspired by model selection and recent results of Goldenschluger and Lepski \cite{LepsP,LepsS}. 
\end{abstract}


\tableofcontents
\newpage
\section{Introduction}
\subsection{Motivation}
Structured models have long served as a representative deterministic model used to describe the evolution of biological systems, see for instance \cite{Pe} or \cite{MD} and references therein. In their simplest form, structured models  describe the temporal evolution of a population structured by a biological parameter such as size, age or any significant \emph{trait}, by means of an evolution law, which is a mass balance at the  macroscopic scale. 
A paradigmatic example is given by the transport-fragmentation equation in cell division, that reads
\begin{equation} \label{equation de transport-fragmentation}
\left\{
\begin{array}{l}
\displaystyle \frac{\partial}{\partial t} n(t, x) + \frac{\partial}{\partial x}\big(g_0(x)n(t,x)\big) + B(x)n(t,x) = 4B(2x)n(t,2x),\;\;t \geq 0,\;\;x \geq 0, \\ \\
g n(t,x=0) = 0,\;\; t>0,\\ \\
n(t=0,x) = n^{0}(x),\;\;x \geq 0.
\end{array}
\right.
\end{equation}

The mechanism captured by Equation \eqref{equation de transport-fragmentation} can be described as a mass balance equation (see \cite{Banks1,MD}): the quantity of cells $n(t,x)$ of size $x$ at time $t$ is fed by a transport term $g_0(x)$ that accounts for growth by nutrient uptake, and each cell can split into two offsprings of the same size according to a division rate $B(x)$.  Supposing $g_0(x)=\kappa g(x),$ where we suppose a given model for the growth rate $g(x)$ known up to a multiplicative constant $\kappa>0,$ and experimental data for $n(t,x),$ 
the problem we consider here is to recover the division rate $B(x)$ and the constant $\kappa$. 

In \cite{PZ}, Perthame and Zubelli proposed a deterministic method based on  the asymptotic behavior of the cell amount $n(t,x):$ indeed, it is known (see \emph{e.g.} \cite{MMP,PR}) that under suitable assumptions on $g$ and $B$, by the use of the \emph{general relative entropy principle} (see \cite{Pe}), one has 
\begin{equation} \label{entropy}
\int_0^\infty \bigl|n(t,x)e^{-\lambda t}-\langle n^0,\phi\rangle{N}(x)\bigr|\phi(x)\,dx\underset{t\to\infty}{\longrightarrow}0
\end{equation}
where $\langle n^0,\phi\rangle=\int n^0(y)\phi(y)dy$ and $\phi$ is the adjoint eigenvector (see \cite{MMP}). The density $N$ is the first eigenvector,  and $(\lambda,N)$ the unique solution of the following eigenvalue problem
\begin{equation} \label{fundamental equation}
\left\{
\begin{array}{l}
\displaystyle \kappa\frac{\partial}{\partial x}\big(g(x)N(x)\big)+\lambda N(x)=4BN(2x) - BN(x),\;\;x>0,\\ \\
B(0)N(0)=0,\qquad \int N(x) dx =1,\qquad N(x)\geq 0,\qquad \lambda>0.
\end{array}
\right.
\end{equation}
Moreover, under some supplementary conditions, this convergence occurs exponentially fast (see \cite{PR}).
Hence, in the rest of this article, we work under the following analytical assumptions.
\begin{assumption} [Analytical assumptions]\label{as:an} 
\
\begin{enumerate}
\item For the considered nonnegative functions $g$ and  $B$ and for $\kappa>0$, there exists a unique eigenpair $(\lambda,\,N)$  solution of Problem \eqref{fundamental equation}. \label{as:an:1}
\item This solution satisfies, for all $p\geq 0,$ $ \int x^p N(x) dx <\infty$ and $0<\int g(x) N(x) dx <\infty$. \label{as:an:2}
\item The functions $N$ and $gN$ belong to $ \mathcal{W}^{s+1}$ with $s\geq 1$, and in particular $\norme{N}_\infty<\infty$ and $\norme{(gN)'}_2<\infty$. ($\mathcal{W}^{s+1}$ denotes the Sobolev space of regularity $s+1$ measured in $\mathbb{L}^2$-norm.)  \label{as:an:3}
\item We have $g\in \mathbb{L}^\infty (\R_+)$ with $\R_+=[0,\infty)$. \label{as:an:4} 
\end{enumerate}
\end{assumption}
Hereafter $\|\cdot\|_2$ and $\|\cdot\|_\infty$ denote the usual $\mathbb{L}^2$ and $\mathbb{L}^\infty$ norms on $\R_+$. Assertions \ref{as:an:1} and \ref{as:an:2} are true under the assumptions on $g$ and $B$ stated in Theorem~1.1 of \cite{DG}, under which we also have $N\in \mathbb{L}^\infty.$ Assertion  \ref{as:an:3} is a (presumably reasonable) regularity assumption, necessary to obtain rates of convergence together with the convergence of the numerical scheme. Assertion \ref{as:an:4} is restrictive, but mandatory in order to apply our statistical approach.

Thanks to this asymptotic behavior provided by the entropy principle \eqref{entropy}, instead of requiring  time-dependent data $n(t,x),$ which is experimentally less precise and more difficult to obtain, the inverse problem becomes: How to recover $(\kappa,B)$ from observations on $(\lambda,N)$ ? 
In \cite{PZ, DPZ}, as generally done in deterministic inverse problems (see \cite{Engl}), it was supposed that experimental data were pre-processed into an approximation $N_\varepsilon$ of $N$ with an {\it a priori} estimate of the form 
$\|N-N_\varepsilon\| \leq \varepsilon$ for a suitable norm $\|\cdot\|$. Then, recovering $B$ from $N_\varepsilon$ becomes an inverse problem with a certain degree of ill-posedness. From a modelling point of view, this approach suffers from the limitation that knowledge on $N$ is postulated in an abstract and somewhat arbitrary sense, that is not genuinely related to experimental measurements.

\subsection{The statistical approach}

In this paper, we propose to overcome the limitation of the deterministic inverse problems approach by assuming that we have $n$ data, each data being obtained from the measurement of an individual cell picked at random, after the system has evolved for a long time so that the approximation $n(t,x) \approx N(x)e^{\lambda t}$ is valid. This is actually what happens if one observes cell cultures in laboratory after a few hours, a typical situation for  {\it E. Coli} cultures for instance, provided, of course, that the underlying aggregation-fragmentation equation is valid. 

Each data is viewed as the outcome of a random variable $X_i$, each $X_i$ having probability distribution 
$N(x)dx$. 
We thus observe $(X_1,\ldots, X_n),$
with
$$\PP(X_1\in dx_1,\ldots, X_n \in dx_n) = \prod_{i = 1}^nN(x_i)dx_i,$$
and where $\PP(\cdot)$ hereafter denotes probability\footnote{In the sequel, we denote by $\E(\cdot)$ the expectation operator with respect to $\PP(\cdot)$ likewise.}. We assume for simplicity that the random variables $X_i$ are defined on a common probability space $(\Omega, {\mathcal F}, \PP)$ and that they are stochastically independent. Our aim is to build an estimator of $B(x)$, that is a function $x \leadsto \hat B_n(x,X_1,\ldots, X_n)$ that approximates the true $B(x)$  with optimal accuracy and nonasymptotic estimates.  To that end, consider the operator
\begin{equation} \label{def operateur tau}
(\lambda,N) \leadsto \mathfrak{T}(\lambda,N)(x):=\kappa\frac{\partial}{\partial x}\big(g(x)N(x)\big)+\lambda N(x),\;\;x\geq 0.
\end{equation}
From representation \eqref{fundamental equation}, we wish to find $B$, solution to
$\mathfrak{T}(\lambda,N) = {\mathcal L}(BN),$
where 
\begin{equation} \label{def L}
{\mathcal L}\big(\varphi\big)(x):=4\varphi(2x)-\varphi(x),
\end{equation}
based on statistical knowledge of $(\lambda, N)$ only.  Suppose that we have  preliminary estimators $\hat L$ and $\hat N$ of respectively $\mathfrak{T}(\lambda,N)$ and $N,$ and an approximation ${\cal L}^{-1}_k$ of ${\cal L}^{-1}$. Then we can reconstruct $B$ in principle by setting formally $$\widehat B:= \frac{{\mathcal L}^{-1}_k(\hat L)}{\widehat N}.$$
This leads us to distinguish three steps that we briefly describe here.  The whole method is fully detailed in Section \ref{proposed-section}.

\

The first and principal step is to find an optimal estimator $\hat L$  for $\mathfrak{T}(\lambda,N).$ To do so, the main part consists in applying twice the Goldenschluger and Lepski's method \cite{LepsS} (GL for short).  This method is a new version of the classical Lepski method \cite{L1, L2, L3, L4}. Both methods are adaptive to the regularity of the unknown signal and the GL method furthermore provides with an oracle inequality. For the unfamiliar reader, we discuss {\it adaptive properties} later on, and explain in details the {\it GL method} and {\it the oracle point of view}  in Section~2. 
\begin{enumerate}
\item First, we estimate the density $N$ by a kernel method, based on a kernel function $K$. We define  $\hat N=\hat N_ {\hat h}$ where $\hat N_h$ is defined by \eqref{def:hatNh} and the bandwidth $\hat h$ is selected automatically by \eqref{selectN} from a properly-chosen set $\cal H.$ (see Section \ref{sec:N:Lepski} for more details). A so-called oracle inequality is obtained in Proposition \ref{estN} measuring the quality of estimation of $N$ by $\hat N$. Notice that this result, which is just a simplified version of \cite{LepsS},  is valid for estimating any density,  since we have only assumed to observe an $n-$sample of $N,$ so that this result 
can be considered \emph{per se.} 
\item Second, we estimate the density derivative (up to $g$) $D=\f{\p}{\p x} (gN)$, again by a kernel method with the same kernel $K$ as before, and select an optimal bandwidth $\tilde{h}$  given by Formula \eqref{def:tildeh} similarly. This defines an estimator $\hat D:=\hat D_{\tilde h}$ where $\hat{D}_h$ is specified by \eqref{def:hatDh}, and yields an oracle inequality for $\hat D$ stated in Proposition \ref{estD}. In the saemway as for $N,$ this result has an interest \emph{per se} and is not a direct consequence of \cite{LepsS}.  
\end{enumerate}
From there, it only remains to find estimators of $\lambda$ and $\kappa$. To that end, we make the following \emph{a priori} (but presumably reasonable) Assumption \ref{HypLa} on the existence of an estimator $\hat\lambda_n$ of $\lambda.$ 
\begin{assumption}[Assumption on $\hat\lambda_n$]\label{HypLa}
There exists some $q>1$ such that
$$\eps_{\lambda,n}  =\big( \E[|\hat\lambda_n-\lambda|^{q}]\big)^{1/q} <\infty, \qquad R_{\lambda,n}  = \E[\hat\lambda_n^{2q}]<\infty.$$
\end{assumption}
Indeed, in practical cell culture experiments, one can track $n$ individual cells that have been picked at random through time. By looking at their evolution, it is possible to infer $\lambda$ in a classical parametric way, via an estimator $\hat{\lambda}_n$ that we shall assume to possess from now on\footnote{Mathematically sepaking, this only amounts to enlarge the probability space to a rich enough structure that captures this estimator. We do not pursue that here.}. 
Based on the following simple equality
\begin{equation}\label{rho}
\kappa = \lambda \rho_g(N) \mbox{ where }  \rho_g(N) = \frac{\int_{\R_+}xN(x)dx}{\int_{\R_+}g(x)N(x)dx},
\end{equation}
obtained by multiplying \eqref{fundamental equation} by $x$ and integrating by part, we then define an estimator $\hat\kappa_n$ by \eqref{def:hatkappan}. Finally, defining $\hat L= \hat\kappa_n \hat D + \hat\lambda_n \hat N$ ends this first step. The second step consists in the formal inversion of ${\mathcal L}$ and its numerical approximation: For this purpose, we follow the method proposed in \cite{DPZ} and recalled in Section \ref{sec:inversion}. To estimate $H:=BN$, we state 
\begin{equation}\label{def:hatH}
\hat H:= {\mathcal L}_k^{-1} (\hat L)
\end{equation}
where ${\mathcal L}_k^{-1}$ is defined by \eqref{def algo} on a given interval $[0,T].$ A new approximation result between ${\mathcal L}^{-1}$ and ${\mathcal L}_k^{-1}$ is given by Proposition~\ref{stabiliteL2}. The third and final step consists in setting
$\hat B:=\f{\hat H}{\hat N},$ 
clipping this estimator in order to avoid explosion when $N$ becomes too small, finally obtaining
\begin{equation}\label{def:Btilde}
\tilde B(x) := \max(\min( \hat{B}(x), \sqrt{n}), -\sqrt{n}).
\end{equation}
\subsection{Rates of convergence}
 Because of the approximated inversion of $\mathcal{L}$ on $[0,T]$, we will have access to error bounds only on $[0,T]$. We set $\norme{f}_{2,T}^2=\int_0^T f^2(x)dx.$ for the $\mathbb{L}^2$-norm restricted to the interval $[0,T]$.
If the fundamental (yet technical) statistical result is the oracle inequality for $\hat H$ stated in Theorem \ref{oraclesurH} (see Section \ref{oracle-section}), the relevant part with respect to existing works in the non-stochastic setting \cite{DPZ,PZ} is its consequence in terms of rates of convergence. For presenting them, we need to assume that the kernel $K$ has regularity and vanishing moments properties.

\begin{assumption}[Assumptions on $K$]\label{HypK2}
The kernel $K$ is differentiable with derivative $K'$. Furthermore,  $\int K(x) dx =1$ and  $\norme{K}_2$ and $\norme{K'}_2$ are finite.
Finally, there exists a positive integer $m_0$ such that $\int K(x)x^pdx=0$ for $p=1,\ldots,m_0-1$ and $I(m_0):=\int |x|^{m_0}K(x)dx$ is finite.
\end{assumption}
Then our proposed estimators satisfy the following properties.

\begin{prop}\label{rate}
Under Assumptions \ref{as:an}, \ref{HypLa} and \ref{HypK2}, let us assume that $R_{\lambda,n}$ and $\sqrt{n} \epsilon_{\lambda,n}$ are bounded uniformly in $n$ and specify ${\cal L}^{-1}_k$  with $k=n$ f. Assume further that the family of bandwidth $\mathcal{H}=\tilde{\mathcal{H}}=\{D^{-1}: \ D=D_{\min},..., D_{\max}\}$ depends on $n$ is such that $1\leq D_{\min}\leq n^{1/(2m_0+1)}$ and $n^{1/5}\leq D_{\max} \leq n^{1/2}$ for all $n$. Then $\hat{H}$ satisfies, for all $s\in[1;m_0-1]$
\begin{equation}\label{0}
\E\big[\norme{\hat H -H}_{2,T}^q\big]=O\big(n^{-\frac{qs}{2s+3}}\big),
\end{equation}
Furthermore, if the kernel K is Lipschitz-regular, if there exists an interval $[a,b]$ in $(0,T)$ such that 
$$[m, M]  :=[\inf_{x\in[a,b]} N(x), \sup_{x\in [a,b]} N(x)] \subset (0,\infty), \qquad Q  :=\sup_{x\in [a,b]} |H(x)|<\infty,$$
and if $\ln(n)\leq D_{\min}\leq n^{1/(2m_0+1)}$ and $n^{1/5}\leq D_{\max} \leq (n/\ln(n))^{1/(4+\eta)}$ for some $\eta>0$, then $\hat{B}$ satisfies, for all $s\in[1,m_0-1]$,
\begin{equation}\label{1}
\E\big[\norme{(\tilde B-B)1_{[a,b]}}^q_{2}\big]=O\big(n^{-\frac{qs}{2s+3}}\big).
\end{equation}
\end{prop}

\subsection{Remarks and comparison to other works}
1) Let us  establish formal correspondences between the methodology and results when recovering $B$ from (\ref{fundamental equation}) from the point of view of statistics or PDE analysis. <
After renormalization,  we obtain the rate $n^{-s/(2s+3)}$ for estimating $B$, and this corresponds to ill-posed inverse problems of order $1$ in nonparametric statistics. We can make a parallel with additive deterministic noise following Nussbaum and Pereverzev \cite{NP} (see also \cite{MR} and the references therein). Suppose we have an approximate knowledge of $N$ and $\lambda$ up to deterministic errors $\zeta_1 \in \mathbb{L}^2$ and $\zeta_2 \in \R$ with noise level $\varepsilon>0$: we observe
\begin{equation} \label{det error model}
N_\varepsilon=N+\varepsilon \zeta_1,\;\;\|\zeta_1\|_{2}\leq 1,
\end{equation}
and $\lambda_\varepsilon = \lambda + \varepsilon \zeta_2,\;\;|\zeta_2|\leq 1$. From the representation
$$B = \frac{{\mathcal L}^{-1}\mathfrak{T}(N,\lambda)}{N},$$
where $\mathfrak{T}(N,\lambda)$ is defined in \eqref{def operateur tau}, we have that the recovery of $\mathfrak{T}(N,\lambda)$ is ill-posed in the terminology of Wahba \cite{W} for it involves the computation of the derivative of $N$. Since ${\mathcal L}$ is bounded with an inverse bounded in $\mathbb{L}^2$ and the dependence in $\lambda$ is continuous, the overall inversion problem is ill-posed of degree $a=1$. By classical inverse problem theory for linear cases\footnote{although here the problem is nonlinear, but that will not affect the argument.}, this means that if $N \in \mathcal{W}^s$, the optimal recovery rate in $\mathbb{L}^2$-error norm should be $\varepsilon^{s/(s+a)} = \varepsilon^{s/(s+1)}$ (see also the work of Doumic, Perthame and collaborators \cite{PZ, DPZ}).

Suppose now that we replace the deterministic noise $\zeta_1$ by a random Gaussian {\it white noise}: we observe
\begin{equation} \label{stochastic error model}
N_\varepsilon = N+\varepsilon \mathbb{B}
\end{equation}
where $\mathbb{B}$ is a Gaussian white noise, {\it i.e.} a random distribution in $\mathcal{W}^{-1/2}$ that operates on test functions $\varphi \in \mathbb{L}^2$ and such that $\mathbb{B}(\varphi)$ is a centered Gaussian variable with variance $\|\varphi\|_{2}^2$. Model \eqref{stochastic error model} serves as a representative toy model for most stochastic error models such as density estimation or signal recovery in the presence of noise. Let us formally introduce the $\alpha$-fold integration operator $\mathcal{I}^\alpha$ and the derivation operator $\partial$.
We can rewrite \eqref{stochastic error model} as
$$N_\varepsilon = \mathcal{I}^1 (\partial N) + \varepsilon \mathbb{B}$$
and applying $\mathcal {I}^{1/2}$ to both side, we (still formally) equivalently observe
$$Z_\varepsilon:=\mathcal{I}^{1/2}N_\varepsilon = \mathcal{I}^{3/2} (\partial N)+\varepsilon \mathcal{I}^{1/2}\mathbb{B}.$$
We are back to a deterministic setting, since in this representation, we have that the noise  $\varepsilon \mathcal{I}^{1/2}\mathbb{B}$ is in $\mathbb{L}^2$. In order to recover $\partial N$ from $Z_\varepsilon$, we have to invert the operator ${\mathcal I}^{3/2}$, which has degree of ill-posedness $3/2$. We thus obtain the rate
$$\varepsilon^{s/(s+3/2)} = \varepsilon^{2s/(2s+3)}=n^{-s/(2s+3)}$$
for the calibration $\varepsilon = n^{-1/2}$ dictated by \eqref{stochastic error model} when we compare our statistical model with the deterministic perturbation (see for instance \cite{N} for establishing formally the correspondence $\varepsilon = n^{-1/2}$ is a general setting). This is exactly the rate we find in Proposition \ref{rate}: the deterministic error model and the statistical error model coincide to that extent\footnote{ The statistician reader willl note that the rate $n^{-s/(2s+3)}$ is also the minimax rate of convergence when estimating the derivative of a density, see \cite{GM}.}.

2) The estimators $\hat{H}$ and $\hat{B}$ do not need the exact knowledge of $s$ as an input  to recover this optimal rate of convergence. We just need to know an upper bound $m_0-1$ to choose the regularity of the kernel $K$. This capacity to obtain the optimal rate without knowing the precise regularity  is known in statistics as adaptivity in the minimax sense (see \cite{tsy} for instance for more details). It is close in spirit to what the discrepancy principle can do in deterministic inverse problems \cite{Engl}. However, in the deterministic framework, one needs to know the level of noise $\varepsilon$, which is not  realistic in practice. In our statistical framework, this level of noise is linked to the size sample $n$ through the correspondence $\varepsilon = n^{-1/2}$. 

3) Finally, note that the rate is polynomial and no extra-logarithmic terms appear, as it is often the case when adaptive estimation is considered (see \cite{L1, L2, L3, L4}).

\medskip

The next section explains in more details the GL approach and presents our estimators to a full extent, including the fundamental oracle inequalities. It also elaborates on the methodology related to oracle inequality. The main advantage of oracle inequalities is that they hold nonasymptotically (in $n$) and that they guarantee an optimal choice of bandwidth with respect to the selected risk. Section \ref{simulations} is devoted to numerical simulations that illustrate the performance of our method. Proofs are delayed until Section \ref{proofs-section}.


\section{Construction and properties of the estimators}\label{proposed-section}
\subsection{Estimation of $N$ by the GL method}\label{sec:N:Lepski}
We first construct an estimator of $N$. A natural approach is a kernel method, which is all the more appropriate for comparisons with analytical methods (see \cite{DPZ} for the deterministic analogue). The kernel function $K$ should satisfy the following assumption, in force in the sequel.
\begin{assumption}[Assumption on the kernel density estimator] \label{HypK}~
$K:\R\to \R$ is a continuous function such that  $\int K(x)dx =1$ and $\int K^2(x)dx <\infty$.
\end{assumption}
For $h>0$ and $x \in \R$, define
\begin{equation}\label{def:hatNh}
 \hat N_h(x):=\frac{1}{n}\sum_{i=1}^nK_h(x-X_i), 
\end{equation}
where $K_h(x)=h^{-1} K(h^{-1}x).$ Note in particular that $\E(\hat{N}_h)=K_h\star N$, where $\star$ denotes convolution. We measure the performance of $\hat{N}_h$ via its squared integrated error, {\it i.e.} the average $\mathbb{L}^2$ distance between $N$ and $\hat{N}_h$. It is easy to see that
$$\E[\|N-\hat{N}_h\|_2]\leq \norme{N-K_h\star N}_2 + \E[\|K_h\star N - \hat{N}_h\|_2],$$
with
\begin{eqnarray*}
 \E[\|K_h\star N - \hat{N}_h\|_2^2]& =& \frac{1}{n^2} \E\big[\int  \Big[\sum_{i=1}^n \Big(K_h(x-X_i)-\E\big(K_h(x-X_i)\big)\Big)\Big]^2 dx \big]\\
 &=& \frac{1}{n^2} \int  \sum_{i=1}^n \E\Big[\Big(K_h(x-X_i)-\E\big(K_h(x-X_i)\big)\Big)^2\Big] dx\\  &\leq& \frac{1}{n} \E \big[\int K_h^2(x-X_1)dx\big] = \frac{\norme{K_h}_2^2}{n}=\frac{\norme{K}_2^2}{nh}.
 \end{eqnarray*}
Applying the Cauchy-Schwarz inequality, we obtain
 $$\E[\|N-\hat{N}_h\|_2]\leq \|N-K_h\star N\|_2 + \frac{1}{\sqrt{nh}} \norme{K}_2.$$
The first term corresponds to a bias term, it decreases when $h\to 0$. The second term corresponds to a variance term, which increases when $h \rightarrow 0$. If one has to choose $h$ in a family $\mathcal{H}$ of possible bandwidths, the best choice is $\bar{h}$ where 
\begin{equation}
\label{oracleN}
\bar{h}:= \argmin_{h\in \mathcal{H}}\big\{\norme{N-K_h\star N}_2 + \frac{1}{\sqrt{nh}} \norme{K}_2 \big\}.
\end{equation}
This ideal compromise $\bar{h}$ is called the "oracle": it  depends on $N$ and then cannot be used in practice. Hence one wants to find an automatic (data-driven) method for selecting this bandwidth. The Lepski method \cite{L1,L2,L3,L4} is one of the various theoretical adaptive methods available for selecting a density estimator. In particular it is the only known method able to select a bandwidth for kernel estimators. However the method do not usually provide a non asymptotic oracle inequality.  Recently, Goldenschluger and Lepski \cite{LepsP} developed powerful probabilistic tools  that enable to overcome this weakness and that can provide with a fully data-driven bandwidth selection method. We give here a practical illustration of their work:  how should one select the bandwidth for a given kernel in dimension 1?

The main idea is to estimate the bias term by looking at several estimators. The method consists in setting first, for any $x$ and any $h,h'>0$,
$$\hat N_{h,h'}(x):=\frac{1}{n}\sum_{i=1}^n(K_h\star K_{h'})(x-X_i)=(K_h\star \hat N_{h'})(x).$$
Next, for any $h\in\HH$, define 
\begin{eqnarray*}
A(h)&:=&\sup_{h'\in\HH}\big\{\|\hat N_{h,h'}-\hat N_{h'}\|_2-\frac{\chi}{\sqrt{nh'}}\|K\|_2\big\}_+\\
&=&\sup_{h'\in\HH}\big\{\max\big\{0,\|\hat N_{h,h'}-\hat N_{h'}\|_2-\frac{\chi}{\sqrt{nh'}}\|K\|_2\big\}\big\},
\end{eqnarray*}
where, given $\e>0$, we set
$\chi:=(1+\e)(1+\|K\|_1).$ The quantity $A(h)$ is actually a good estimator of $\norme{N-K_h\star N}_2$ up to the term $\norme{K}_1$ (see (\ref{Ah*}) and (\ref{biais}) in Section~\ref{proofs-section}). The next step consists then in setting
\begin{equation}\label{selectN}
\hat h:=\arg\min_{h\in\HH}\big\{A(h)+\frac{\chi}{\sqrt{nh}}\|K\|_2\big\},
\end{equation}
and our final estimator of $N$ is obtained by putting $ \hat N:=\hat N_{\hat h}$. Let us specify what we are able to prove at this stage.
\begin{prop}\label{estN}
Assume $N \in \mathbb{L}^\infty$ and work under Assumption \ref{HypK}. If $\HH\subset\{D^{-1},D=1,\ldots, D_{\max}\}$ with $D_{\max}=\delta n$ for $\delta>0$, then, for any $q\geq 1 $,
 $$\E\big[\|\hat N-N\|_2^{2q}\big]\leq C(q) \chi^{2q} \inf_{h\in \HH}\big\{\|K_{h}\star N-N\|_2^{2q}+\frac{\norme{K}_2^{2q}}{(hn)^q}\big\}+C_1n^{-q},$$
 where $C(q)$ is a constant depending on $q$ and $C_1$ is a constant depending on $q,$ $\ve,$ $\delta,$ $\|K\|_2,$ $\|K\|_1$  and $\|N\|_\infty$.
 \end{prop}
The previous inequality is called an oracle inequality, for we have $\E[\|\hat N-N\|_2]\leq (\E[\|\hat N-N\|_2^{2q}])^{1/(2q)}$ and $\hat{h}$ is performing as well as the oracle $\bar{h}$ up to some multiplicative constant. In that sense, we are able to select the {\it best bandwidth} within our family $\HH$. 
\begin{remark}
As compared to the results of Goldenschluger and Lepski in \cite{LepsP}, we do not consider the case where $\HH$ is an interval and  we do not specify $K$ except for Assumption \ref{HypK}. This simpler method is more reasonable from a numerical point of view, since estimating $N$ is only a preliminary step. The probabilistic tool we use here is classical in model selection theory (see Section~\ref{proofs-section} and \cite{mas}) and actually, we do not use directly \cite{LepsP}. In particular the main difference is that, in our specific case, we are able to get $\max(\HH)$ fixed whereas Goldenschluger and Lepski \cite{LepsP} require $\max(\HH)$ to tend to 0 with $n$. The price to pay is that we obtain a uniform bound (see Lemma \ref{concentration} in Section \ref{technique}) which is less tight, but that will be sufficient for our purpose.
\end{remark}
\subsection{Estimation of $\frac{\partial}{\partial x}\big(g(x)N(x)\big)$ by the GL method}
The previous method can of course be adapted to estimate 
$$D(x):=\frac{\partial}{\partial x}\big(g(x)N(x)\big).$$ 
We adjust here the work of \cite{LepsS} to the setting of estimating a derivative. We again use kernel estimators with more stringent assumptions\footnote{For sake of simplicity we use the same kernel to estimate $N$ and $D$ but this choice is not mandatory.}
 on $K$. 
 \begin{assumption}[Assumption on the kernel of the derivative estimator] \label{HypK'}~
The function $K$ is differentiable, $\int K(x)dx =1$ and $\int (K'(x))^2dx <\infty$.
\end{assumption}

For any bandwidth $h >0$,  we define the kernel estimator of $D$ as
\begin{equation}\begin{array}{lll}\label{def:hatDh}
\hat D_h(x)&:=&\frac{1}{n}\sum_{i=1}^ng(X_i)K_h'(x-X_i)\\
 &=&\frac{1}{nh^2}\sum_{i=1}^ng(X_i)K'\big(\frac{x-X_i}{h}\big).\end{array}
\end{equation}
Indeed
\begin{eqnarray}
\E(\hat D_h(x))&=&\int K_h'(x-u)g(u)N(u)du\nonumber\\
&=&\big(K_h'\star (gN)\big)(x)
=\big(K_h\star (gN)'\big)(x).\nonumber
\end{eqnarray}
Again we can look at the integrated squared error of $\hat{D}$.  We obtain the following upper bound:
$$\E[\|\hat{D}_h-D\|_2] \leq \|D-K_h\star D\|_2 + \E[\|K_h\star D - \hat{D}_h\|_2],$$
with
\begin{eqnarray*}
 \E[\|K_h\star D - \hat{D}_h\|_2^2 &= &\frac{1}{n^2} \E\Big[\int  \Big[\sum_{i=1}^n \Big(g(X_i)K'_h(x-X_i)-\E\big(g(X_i)K'_h(x-X_i)\big)\Big)\Big]^2 dx \Big]\\
 &=& \frac{1}{n^2} \int  \sum_{i=1}^n \E\Big[\Big(g(X_i)K'_h(x-X_i)-\E\big(g(X_i)K'_h(x-X_i)\big)\Big)^2\Big] dx\\ & \leq &\frac{1}{n} \E \big[\int g^2(X_1){K'_h}^2(x-X_1) dx\big]\\&\leq& \frac{\norme{g}_\infty^2\norme{K'_h}_2^2}{n}=\frac{\norme{g}_\infty^2\norme{K'}_2^2}{nh^3}.
 \end{eqnarray*}
Hence,  by Cauchy-Schwarz inequality
 $$\E[\|D-\hat{D}_h\|_2]\leq \norme{D-K_h\star D}_2 + \frac{1}{\sqrt{nh^3}} \norme{g}_\infty\norme{K'}_2.$$
Once again, there is a bias-variance decomposition, but now the variance term is of order $\frac{1}{\sqrt{nh^3}} \norme{K'}_2\norme{g}_\infty$. We therefore define the oracle by
\begin{equation}
\bar{\bar h}:= \argmin_{h\in \tilde{\mathcal{H}}}\big\{\norme{D-K_h\star D}_2 + \frac{1}{\sqrt{nh^3}} \norme{g}_\infty \norme{K'}_2 \big\}.
\end{equation}
Now let us apply the GL method in this case. Let $\tilde\HH$ be a family of bandwidths.
We set for any $h,h'>0$,
\begin{eqnarray*}
 \hat D_{h,h'}(x)&:=&\frac{1}{n}\sum_{i=1}^ng(X_i)(K_{h}\star K_{h'})'(x-X_i)
\end{eqnarray*}
and
\begin{eqnarray*}
\tilde A(h)&:=&\sup_{h'\in\tilde\HH}\big\{\|\hat D_{h,h'}-\hat D_{h'}\|_2-\frac{\tilde\chi}{\sqrt{n h'^3}}\|g\|_\infty\|K'\|_2\big\}_+,
\end{eqnarray*}
where, given $\tilde\e>0$, we put
$\tilde\chi:=(1+\tilde\e)(1+\|K\|_1).$
Finally, we estimate $D$ by using $\hat D:=\hat D_{\tilde h}$ with
\begin{equation}\label{def:tildeh}
\tilde h:=\argmin_{h\in\tilde\HH}\big\{\tilde A(h)+\frac{\tilde\chi}{\sqrt{nh^3}}\|g\|_\infty\|K'\|_2\big\}.
\end{equation}
As before, we are able to prove an oracle inequality for $\hat{D}$.
\begin{prop}\label{estD}
Assume $N \in \mathbb{L}^\infty$. Work under Assumption \ref{HypK'}.
If $\tilde\HH=\{D^{-1}, D=1 \ldots, \tilde D_{\max}\}$, with $\tilde D_{\max}=\sqrt{\tilde\delta n}$ for $\tilde\delta>0$, then for any $q\geq 1 $,
 $$\E\big[\|\hat D-D\|_2^{2q}\big]\leq \tilde C(q)\tilde{\chi}^{2q} \inf_{h\in \tilde \HH}\big\{\|K_{h}\star D-D\|_2^{2q}+\big(\frac{\norme{g}_\infty\|K'\|_2}{\sqrt{nh^3}}\big)^{2q}\big\}+\tilde C_1n^{-q},$$
 where $\tilde C(q)$ is a constant depending on $q$ and $\tilde C_1$ is a constant depending on $q,$ $\tilde\ve,$ $\tilde\delta,$ $\norme{K'}_2,$ $\norme{K'}_1$, $\norme{g}_\infty$  and $\norme{N}_\infty$.
\end{prop}
\subsection{Estimation of $\kappa$ (and $\lambda$)}\label{lambda-section}
As mentioned in the introduction, we will not consider the problem of estimating $\lambda$ and we work under Assumption \ref{HypLa}: an estimator $\hat{\lambda}_n$ of $\lambda$ is furnished by the practitioner prior to the data processing for estimating $B$. It becomes subsequently straightforward to obtain an estimator of $\kappa$ by estimating $ \rho_g(N)$, see the form of \eqref{rho}. We estimate $\rho_g(N)$ by
\begin{equation}\label{def:hatrhon}
\hat\rho_n:=\frac{\sum_{i=1}^nX_i}{\sum_{i=1}^ng(X_i)+c},\end{equation}
where $c >0$ is a (small) tuning constant\footnote{In practice, one can take $c=0$.}.
Next we simply put 
\begin{equation}\label{def:hatkappan}
\hat\kappa_n= \hat\lambda_n\hat\rho_n.
\end{equation}
\subsection{Approximated inversion of ${\mathcal L}$}\label{sec:inversion}
From \eqref{def:hatrhon}, the right-hand side of \eqref{fundamental equation} is consequently estimated by 
$$ \hat\kappa_n\hat D +\hat\lambda_n \hat N.$$
It remains to formally apply the inverse operator $\mathcal{L}^{-1}$. However
Given $\varphi$, the dilation equation 
\begin{equation} \label{inverse L}
{\mathcal L}\big(\psi\big)(x)=4\psi(2x)-\psi(x) = \varphi(x),\;\;x\in \R_+
\end{equation}
admits in general infinitely many solutions, see Doumic {\it et al.} \cite{DPZ}, Appendix A. Nevertheless, if $\varphi \in \mathbb{L}^2$, there is a unique solution $\psi \in \mathbb{L}^2$ to \eqref{inverse L}, see Proposition A.1. in \cite{DPZ}, and moreover it defines a continuous operator ${\cal L}^{-1}$ from $\mathbb{L}^2$ to $\mathbb{L}^2$. Since $K_h$ and $gK_h'$ belong to $\mathbb{L}^2$, one can define a unique solution to \eqref{inverse L} when $\varphi= \hat\kappa_n\hat D +\hat\lambda_n \hat N$. This inverse is not analytically known but we can only approximate it via the fast algorithm described below.

Given $T>0$ and an integer $k\geq 1$, we construct a linear operator ${\mathcal L}_k^{-1}$ that maps a function $\varphi \in {\mathbb L}^2$ into a function with compact support in $[0,T]$ as follows.
Consider the regular grid on $[0,T]$ with mesh $k^{-1}T$ defined by
$$0 =x_{0,k} < x_{1,k} < \cdots < x_{i,k}:=\tfrac{i}{k}T < \ldots < x_{k,k}=T.$$
We set
$$\varphi_{i,k}:=\frac{k}{T}\int_{x_{i},k}^{x_{i+1,k}}\varphi(x)dx\;\;\text{for}\;\;i=0,\ldots, k-1,$$
and define by induction the sequence \footnote{for any sequence $u_i,i=1,2,\ldots$, we define
$$u_{i/2}:=
\big\{
\begin{array}{ll}
u_{i/2} & \text{if}\;i\;\text{is even} \\
\tfrac{1}{2}(u_{(i-1)/2}+u_{(i+1)/2})& \text{otherwise}. 
\end{array}
\big.$$}
$$H_{i,k}(\varphi):=\frac{1}{4}(H_{i/2,k}(\varphi)+\varphi_{i/2,k}),\;\;\;i=0,\ldots, k-1,$$
what gives, for $i=0$ and $i=1$
$$H_0(\varphi):=\tfrac{1}{3}\varphi_{0,k},\;\;H_1(\varphi):=\tfrac{4}{21}\varphi_{0,k}+\tfrac{1}{7}\varphi_{1,k}.$$ 
Finally, we define
\begin{equation} \label{def algo}
{\mathcal L}^{-1}_k(\varphi)(x):=\sum_{i=0}^{k-1} H_{i,k}(\varphi) 1_{[x_{i,k},x_{i+1,k})}(x).
\end{equation}
As stated in the introduction,  we eventually estimate $H=BN$ by
$$\hat H= {\mathcal L}^{-1}_k(\hat\kappa_n\hat D +\hat\lambda_n \hat N).$$
The stability of the inversion is given by the fact that ${\cal L}^{-1}_k: \mathbb{L}^2 \rightarrow \mathbb{L}^2$ is  continuous,  see Lemma \ref{lk-1} in Section \ref{technique}, and by the following approximation result between ${\cal L}^{-1}$ and ${\cal L}^{-1}_k.$
\begin{prop}\label{stabiliteL2}
Let  $T>0$ and $\varphi \in \mathcal{W}^{1}$. Let ${\mathcal L}^{-1}(\varphi)$ denote the unique solution of \eqref{inverse L} belonging to ${\mathbb L}^{2}$. We have  for $k\geq 1$: 
$$\|{\mathcal L}^{-1}_k(\varphi)-{\mathcal L}^{-1}(\varphi)\|_{2,T} \leq C \f{T}{ \sqrt{k}} \|\varphi\|_{ \mathcal{W}^{1}},$$
with  $C < \f{1}{\sqrt{6}}.$
\end{prop}
Hence, ${\mathcal L}^{-1}_k$ behaves nicely over sufficiently smooth functions. Moreover the estimation of $N$ and the estimators $\hat\kappa_n$ and $\hat\lambda_n$ are essentially regular. Finally we estimate $B$ as stated in \eqref{def:Btilde}. The overall behaviour of the estimator is finally governed by the quality of estimation of the derivative $D$, which determines the accuracy of the whole inverse problem in all these successive steps.

\subsection{Oracle inequalities for $\hat H$} \label{oracle-section}

We are ready to state our main result, namely the oracle inequality fulfilled by $\hat{H}$.
\begin{theorem}
\label{oraclesurH}
Work under Assumptions~\ref{as:an} and \ref{HypLa} and
let $K$ a kernel satisfying Assumptions \ref{HypK} and \ref{HypK'}.  Define $\HH\subset\{D^{-1}, D=1 \ldots , D_{\max}\}$ with $D_{\max}=\delta n$ for $\delta>0$ and  $\tilde\HH\subset\{D^{-1}:\ D=1,..., \tilde D_{\max}\}$ with $\tilde D_{\max}=\sqrt{\tilde\delta n}$ for $\tilde\delta>0.$  For $k\geq 1$ and $T>0,$ let us define ${\cal L}^{-1}_k$ by \eqref{def algo} on the interval $[0,T].$ Finally, define the estimator
$$\hat{H}={\cal L}_k^{-1} (\hat\lambda_n \hat N_{\hat h} + \hat \kappa_n \hat D_{\tilde h}),$$ 
where $\hat N_{\hat h}$ and $\hat D_{\tilde h}$ are kernel estimators defined respectively by \eqref{def:hatNh} and \eqref{def:hatDh}, and where we have selected  $\hat h$ and $\tilde h$ by \eqref{selectN} and \eqref{def:tildeh}. Moreover take $\hat\kappa_n$ as defined by \eqref{def:hatrhon} and \eqref{def:hatkappan} for some $c>0.$

The following upper bound holds  for any $n:$
\begin{multline*}
\E\big[\norme{\hat H -H}_{2,T}^q\big]\leq C_1 \big\{ \sqrt{R_{\lambda,n}} \inf_{h\in \tilde\HH}\big[ \|K_{h}\star D-D\|_2^{q}+\big(\frac{\norme{g}_\infty\|K'\|_2}{\sqrt{nh^3}}\big)^{q}\big] \big.\\
\big.+ \inf_{h\in \HH}\big[\|K_{h}\star N-N\|_2^{q}+\big(\frac{\norme{K}_2}{\sqrt{nh}}\big)^q\big]+ \ve_{\lambda,n} ^q+ \big((\|N\|_{{\cal W}^1} +\|g N\|_{{\cal W}^2})\f{T}{\sqrt{k}}\big)^q\big\}+
C_2n^{-\frac{q}{2}},
\end{multline*}
where $C_1$ is a constant depending on $q$, $g$, $N$, $\ve$, $\tilde\ve$, $\norme{K}_1$ and $c$; and $C_2$ is a constant depending on $q,$ $g,$ $N,$ $\tilde\ve,$ $\ve,$ $\delta,$ $\tilde\delta,$ $\norme{K}_2,$ $\norme{K}_1,$ $\norme{K'}_2,$ $\norme{K'}_1$. 
\end{theorem}

1) Note that the upper bound quantifies the additive cost of each step used in the estimation method. 
The first part is an oracle bound on the estimation of $D$ times the size of the estimator of $\lambda$. The second part is  the oracle bound for $N$. Of course, the results are sharp only if $\hat \lambda$ is good, which can be seen through $\ve_{\lambda,n}$. Finally, the bound is also governed by the approximated inversion through the only term where $k$ appears. The last term is just a residual term, that will be in most of the cases negligible with respect to the other terms. In particular since all the previous errors are somehow unavoidable, this means that, as far as our method is concerned, our upper bound is the best possible that can be achieved in order to select the different bandwidths $h$, up to multiplicative constants. Moreover, one can see how the limitation in $k$ influences the method and how large $k$ needs to be chosen to guarantee that the main error comes from the fact that one estimates a derivative. 

2) The result holds for any $n$. In particular, we expect the method to perform well for small sample size, see the numerical illustration below. This also shows that we are able to select a good bandwidth as far as the kernel $K$ is fixed, even if there is no assumption on the moments of $K$ and consequently on the approximation properties of $K$. In the next simulation section, we  focus on a Gaussian kernel which has only one vanishing moment (hence one cannot really consider minimax adaptation for regular function with it) but for which the bandwidth choice is still important in practice. The previous result guarantees an optimal bandwidth choice even for this kernel, up to some multiplicative constant.

3) From this oracle result, we can easily deduce the rates of convergence of Proposition \ref{rate} at the price of further assumptions on the kernel $K$, {\it i.e.} Assumption \ref{HypK2} defined in the Introduction.  Section~1.2.2 of \cite{tsy} recalls how to build compactly supported kernels satisfying Assumption \ref{HypK2}. 
If $m_0$ satisfies Assumption~\ref{HypK2}, then for any $s\leq m_0,$ for any $f\in\mathcal{W}^s$,
$$\|K_{h}\star f-f\|_2\leq C\|f\|_{\mathcal{W}^s}h^{s},$$ where $C$ is a constant that can be expressed by using $K$ and $m_0$ (see Theorem 8.1 of \cite{hkpt}). 
Now, it is sufficient to choose
$h=D^{-1}$ of order $n^{-1/(2s+3)}$ to obtain \eqref{0}. The complete proof of Proposition \ref{rate} is delayed until the last section.
\section{Numerical illustration}\label{simulations}

Let us illustrate our method through some simulations.

\subsection{The numerical protocol}

First, we need to build up simulated data: to do so, we depart from given $g,$ $\kappa$ and $B$ on a regular grid $ [0,dx,\dots,X_M]$ and solve the direct problem by the use of the so-called \emph{power algorithm} to find the corresponding density $N$ and the principal eigenvalue $\lambda$ (see for instance \cite{DPZ} for more details).  We check that $N$ almost vanishes outside the interval $[0,X_M];$ else, $X_M$ has to be increased. The density $N$ being given on  this grid, we approximate it by a spline function, and build a $n$-sample by the rejection sampling algorithm. For the sake of simplicity, we do not simulate an approximation on $\lambda$ and keep the exact value, thus leading, in the estimate of Theorem \ref{oraclesurH}, to $R_{\lambda,n}=\lambda^{2q}$ and $\hat\lambda_n=\lambda.$

\

We then follow step by step the method proposed here and detailed in Section \ref{proposed-section}.
\begin{enumerate}
\item The GL method for the choice of $\hat{N}=\hat{N}_{\bar{h}}$. We take the classical Gaussian kernel $K(x)=(2\pi)^{-1/2}\exp\big(-x^2/2\big)$, set $D_{\max}=n$ and limit ourselves to a logarithmic sampling ${\cal H}=\{1,1/2,\dots,1/9,1/10,1/20,\dots,1/100,1/200,\dots,1/n\}$ in order to reduce the cost of computations (The GL method is indeed the most time-consuming step in the numerical protocol).
\item The GL method for the choice of $\hat{D}_h$. The procedure is similar except that we choose here $D_{max}=\sqrt{n}.$ The selected bandwidths $\bar{h}$ and $\tilde{h}$ can be different. We check that the GL method does not select an extremal point of $\cal H.$
\item The choice of $\hat{\kappa}_n$, as defined by \eqref{def:hatkappan}.
\item The numerical scheme described in Section \ref{sec:inversion} and in \cite{DPZ} for the inversion of $\cal L$.
\item The division by $\hat N$ and definition of $\tilde B$ as described in \eqref{def:Btilde}.
\end{enumerate}
At each step, we compare, in $\mathbb{L}^2$-norm, the reconstructed function and the original one: $\hat N$ \emph{vs} $N,$ $\f{\p}{\p x} (g\hat{N})$ \emph{vs} $\f{\p}{\p x} (gN),$ $\hat{H}$ \emph{vs} $BN$ and finally $\hat{B}$ \emph{vs} B.

\subsection{Results on simulated data}

We first test the three cases simulated in \cite{DPZ} in which the numerical analysis approach was dealt with. Namely on the interval $[0,4]$, we consider the cases where $g\equiv 1$ and first $B=B_1\equiv 1,$ second $B(x)=B_2(x)=1$ for $x\leq 1.5,$ then linear to $B_2(x)=5$ for $x\geq 1.7.$ This particular form is interesting because due to this fast increase on $B,$ the solution $N$ is not that regular and exhibits a 2-peaks distribution (see Figure \ref{fig:B123:1}).
 Finally, we test $B(x)=B_3(x)=\exp(-8(x-2)^2)+1.$ 
\begin{figure}[ht]
\begin{center}
\begin{minipage}{17cm}
\includegraphics[width=7cm, height=7cm]{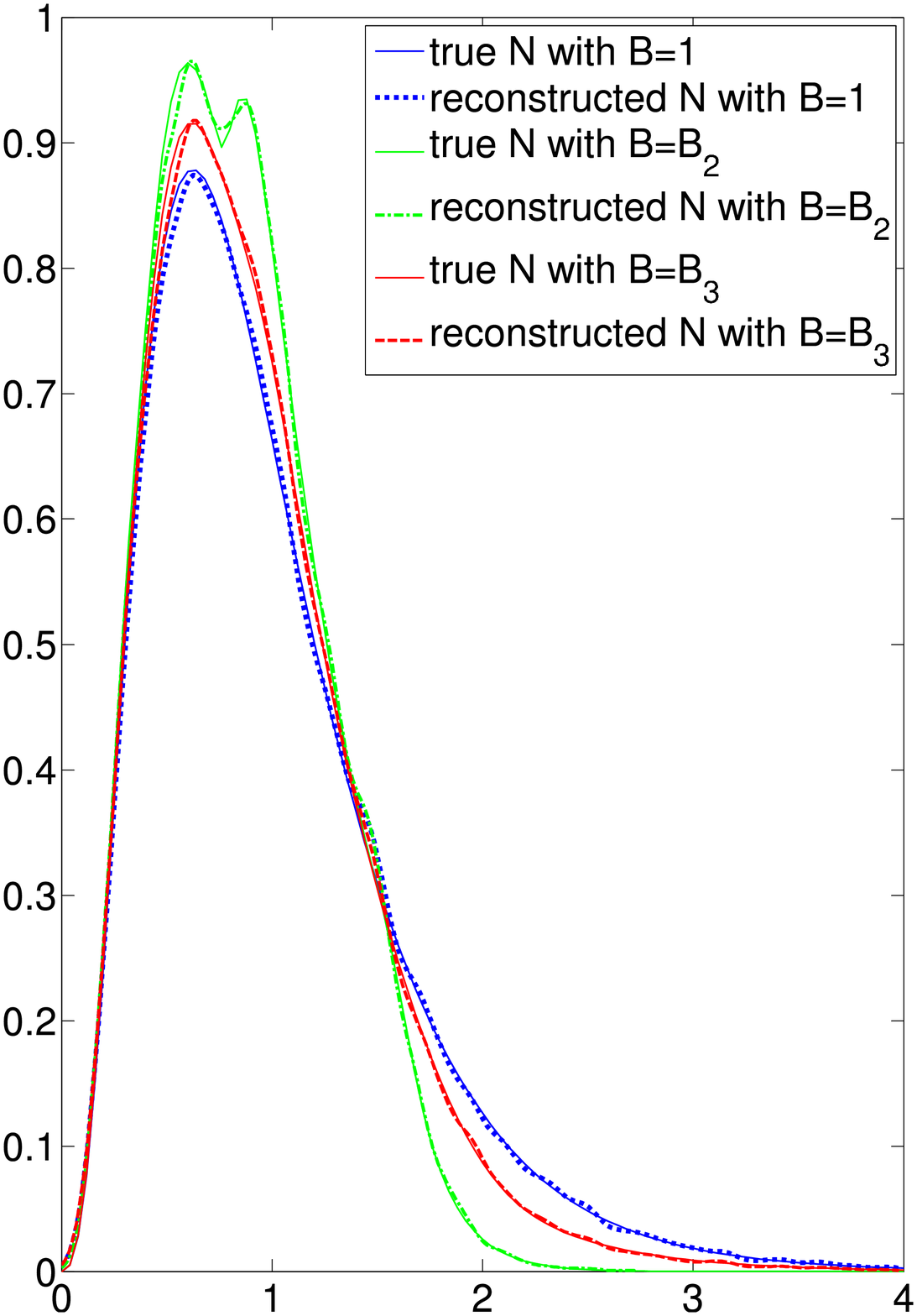} \quad \includegraphics[width=7cm, height=7cm]{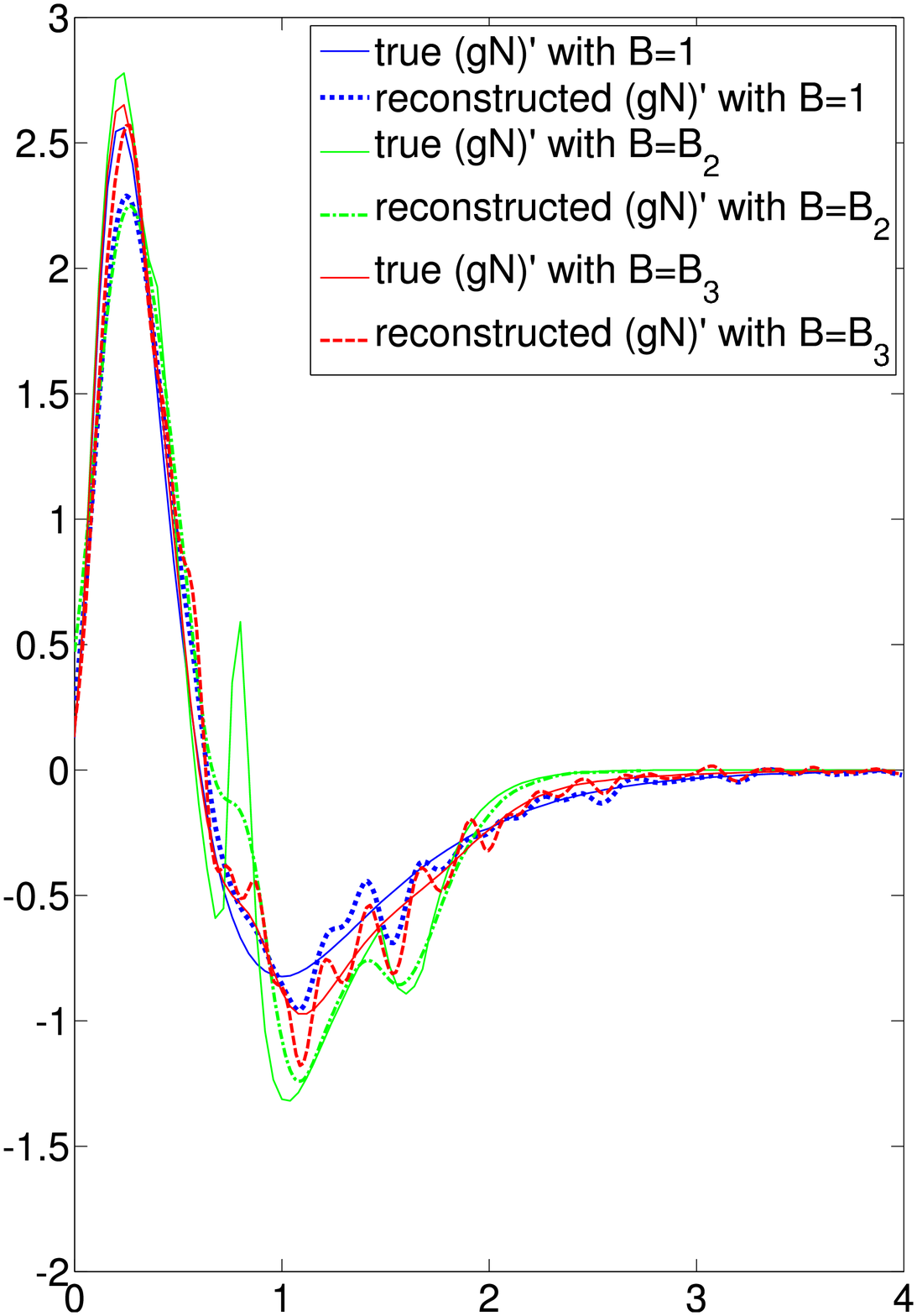}
\end{minipage} \end{center}\vspace{-0.5cm}
\caption{\label{fig:B123:1} Reconstruction of  $N$ (left) and of $\f{\p}{\p x} (gN)$ (right) obtained with a sample of $n=5.10^4$ data, for three different cases of division rates $B.$}
\end{figure}
\begin{figure}[ht]
\begin{center}
\begin{minipage}{17cm}
\includegraphics[width=7cm, height=7cm]{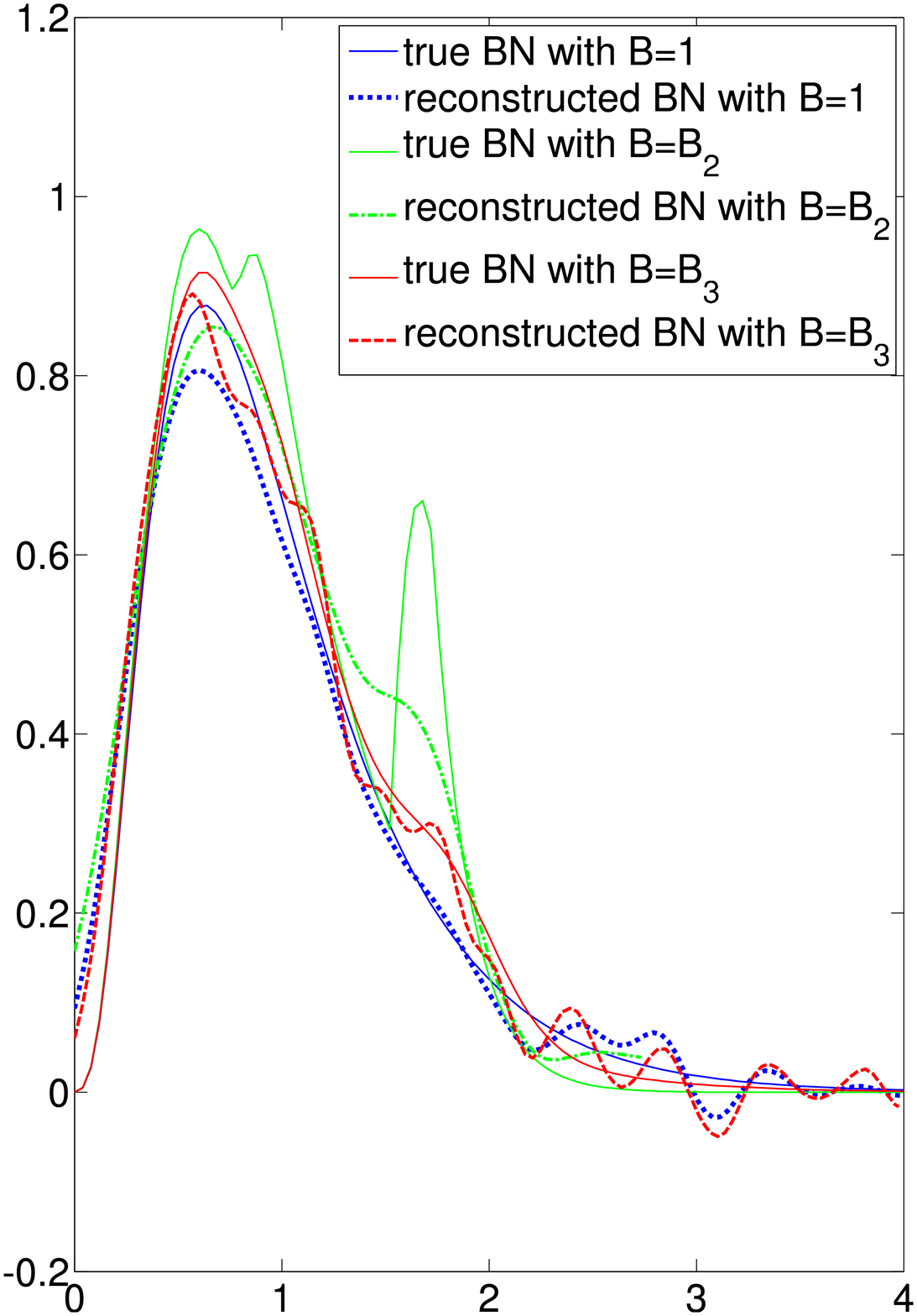} \quad \includegraphics[width=7cm, height=7cm]{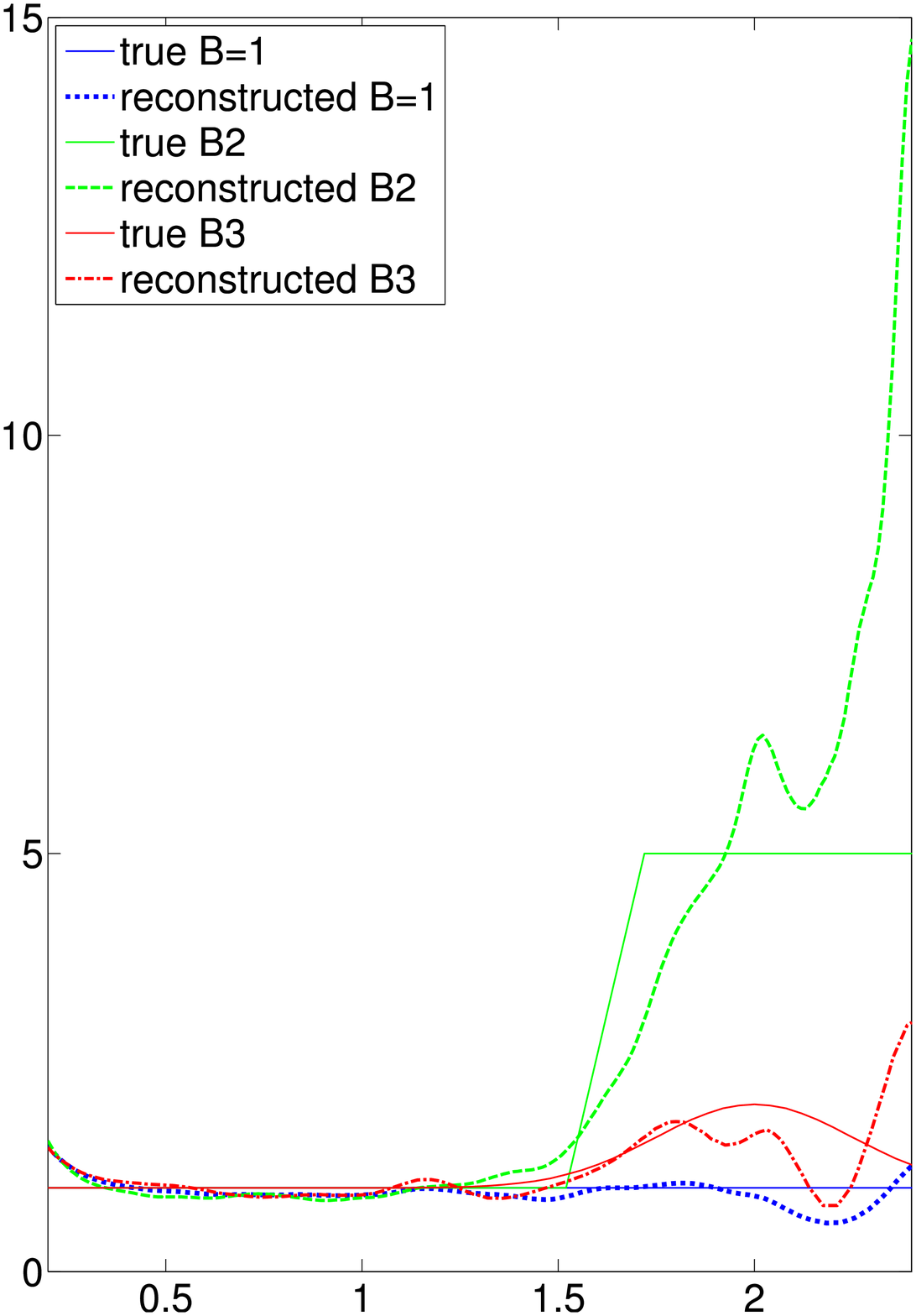}
\end{minipage} 
\end{center}\vspace{-0.5cm}
\caption{\label{fig:B123:2}Reconstruction of  $BN$ (left) and of $B$ (right) obtained with a sample of $n=5.10^4$ data, for three different cases of division rates $B.$ }
\end{figure}

In Figures \ref{fig:B123:1} and \ref{fig:B123:2}, we show the simulation results with $n=5.10^4$ (a realistic value for \emph{in vitro} experiments on {\it E. Coli} for instance) for the reconstruction of $N,$ $\f{\p}{\p x} (gN),$ $BN$ and $B.$

One notes that the solution can well capture the global behavior of the division rate $B,$ but, as expected, has more difficulties in recovering fine details (for instance, the difference between $B_1$ and $B_3$) and also gives much more error when $B$ is less regular (case of $B_2$). One also notes that even if the reconstruction of $N$ is very satisfactory, the critical point is the reconstruction of its derivative.
Moreover, for large values of $x,$ even if $N$ and its derivative are correctly reconstructed, the method fails in finding a proper division rate $B.$ This is due to two facts: first, $N$ vanishes, so the division by $N$ leads to error amplification. Second, the values taken by $B(x)$ for large $x$ have little influence on the solutions $N$ of the direct problem: whatever the values of $B$, the solutions $N$ will not vary much, as shown by Figure \ref{fig:B123:1} (left). A similar phenomenon occurred indeed when solving the deterministic problem in \cite{DPZ} (for instance, we refer to Fig. 10 of this article for a comparison of the results).  
 
We also test a case closer to biological true data, namely the case $B(x)=x^2$ and $g(x)=x.$ The results are shown on Figures \ref{fig:B4:1} and \ref{fig:B4:2} for $n$-samples of size $10^3,$ $5.10^3,$ $10^4$ and $5.10^4.$

\begin{figure}[ht]
\begin{center}
\begin{minipage}{15cm}
\includegraphics[width=7cm, height=7cm]{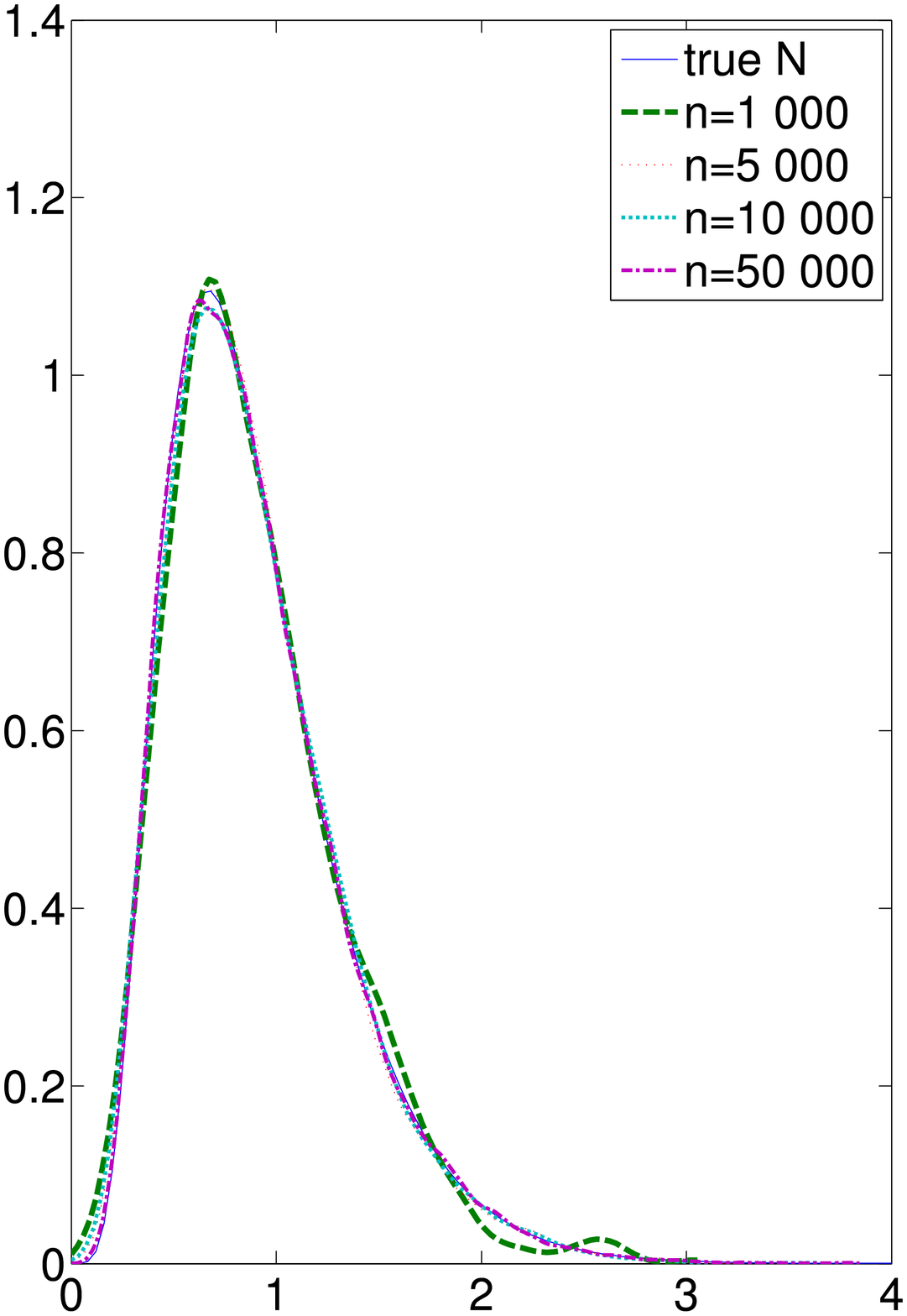} \quad \includegraphics[width=7cm, height=7cm]{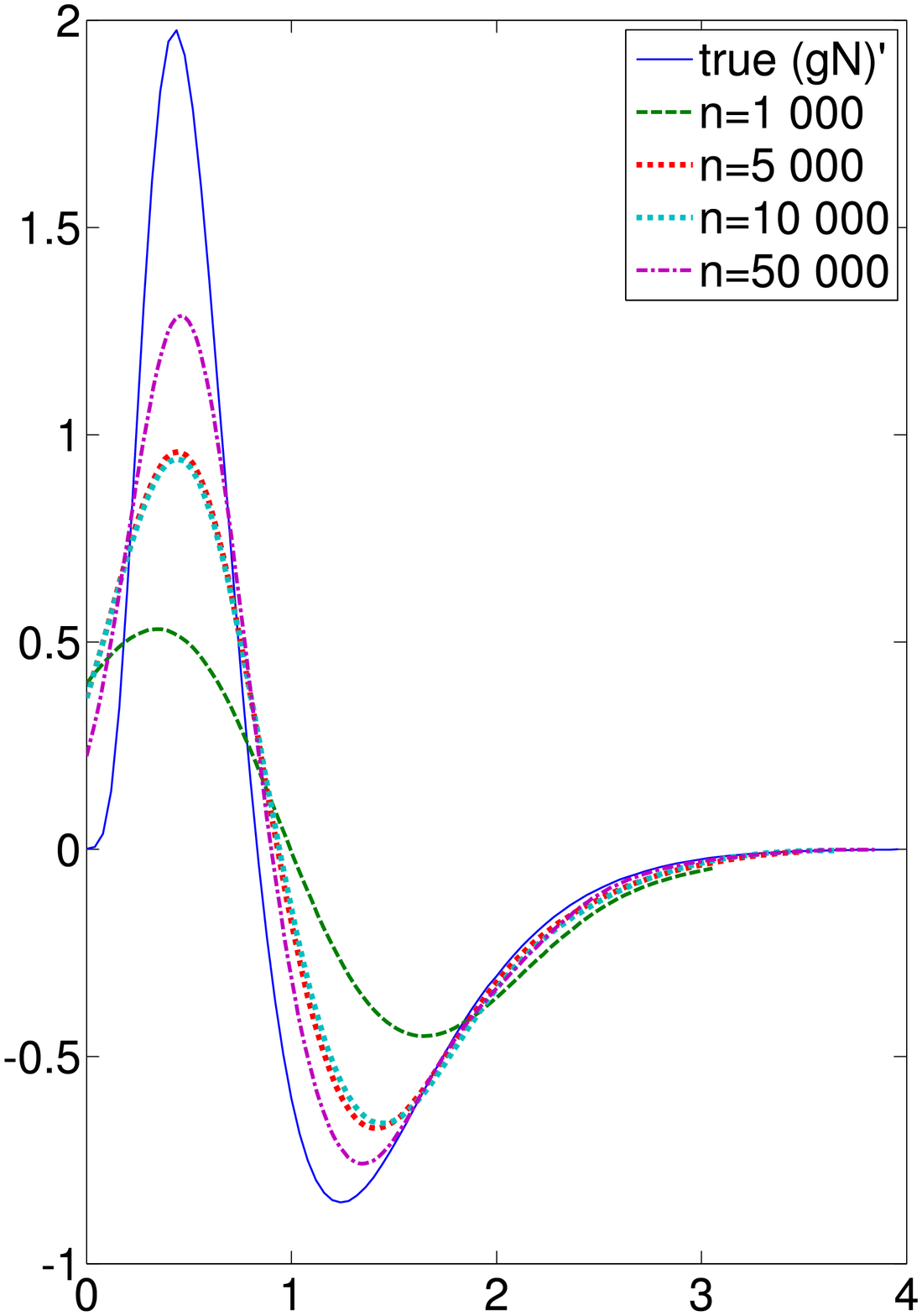}
\end{minipage} \end{center}\vspace{-0.5cm}
\caption{\label{fig:B4:1} Reconstruction of  $N$ (left) and of $\f{\p}{\p x} (gN)$ (right) obtained for $g(x)=x$ and $B(x)=x^2,$ for various sample sizes.}
\end{figure}

\begin{figure}[ht]
\begin{center}
\begin{minipage}{15cm}
\includegraphics[width=7cm, height=7cm]{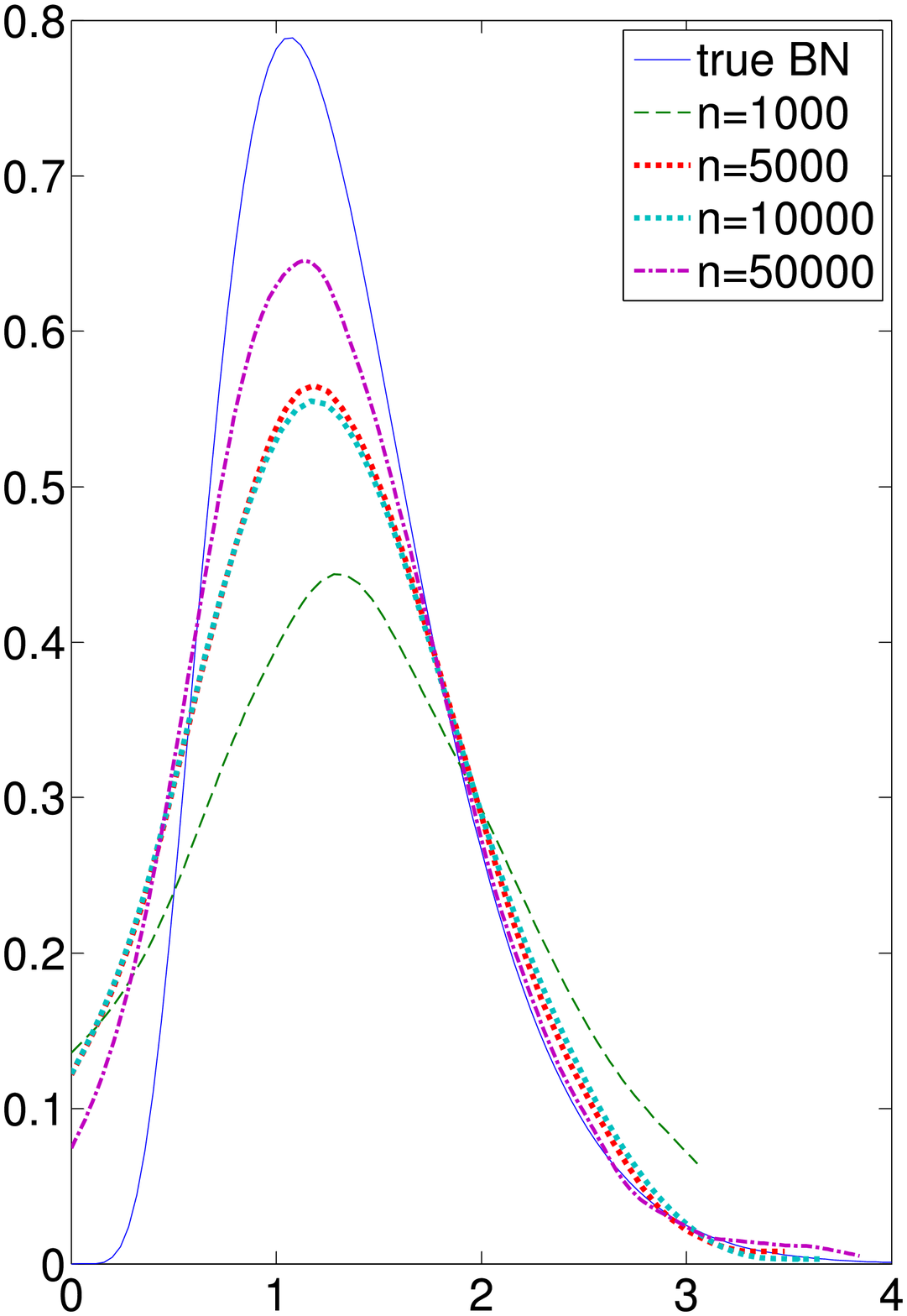} \quad \includegraphics[width=7cm, height=7cm]{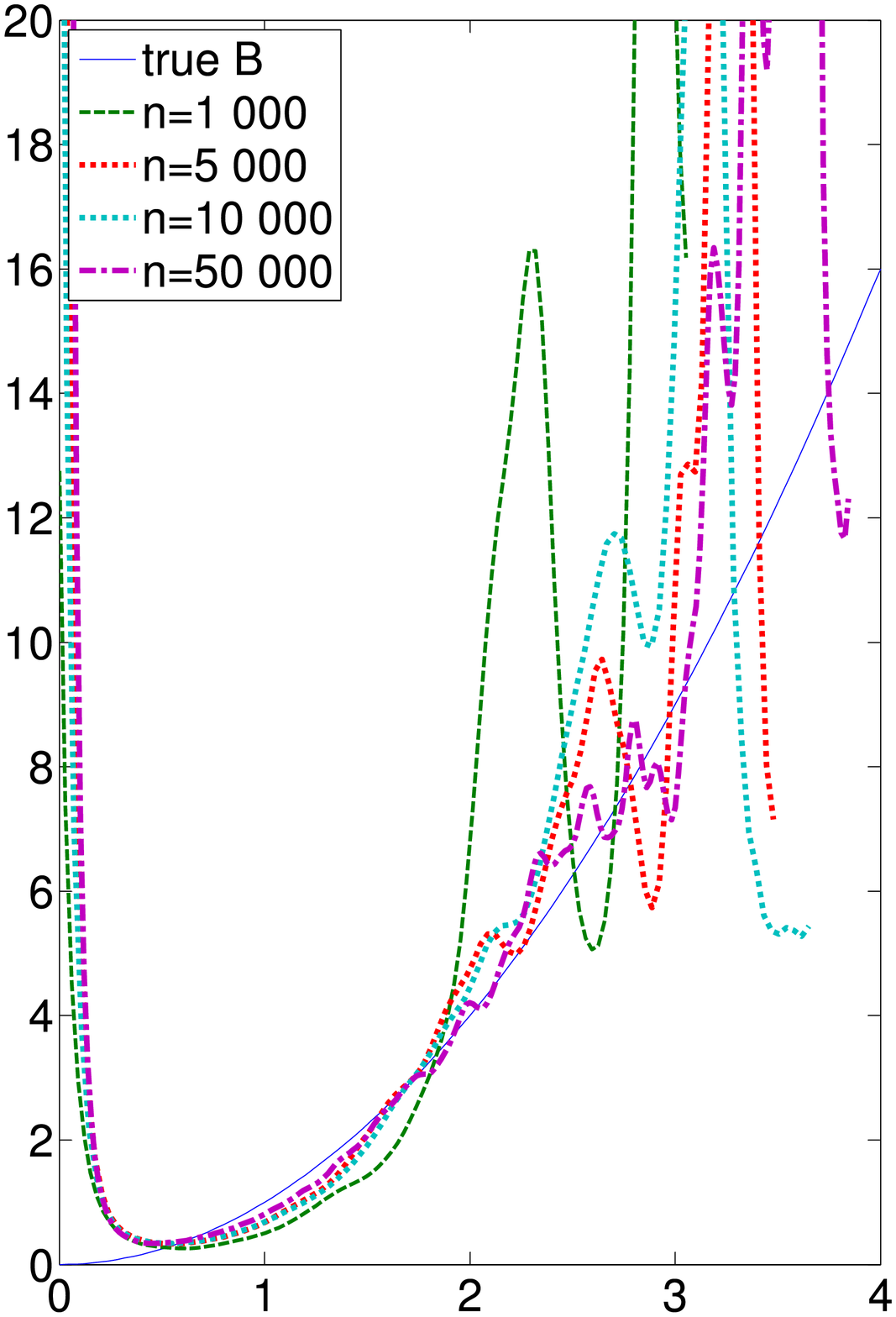}
\end{minipage} \end{center}\vspace{-0.5cm}
\caption{\label{fig:B4:2} Reconstruction of  $BN$ (left) and of $B$ (right) obtained for $g(x)=x$ and $B(x)=x^2,$ for various sample sizes.}
\end{figure}

One notes that reconstruction is already very good for $N$ when $n=10^3,$ unlike the reconstruction of $\f{\p}{\p x} (gN)$ that requires much more data. 

Finally, in Table \ref{tab:num} we give average error results on $50$ simulations, for $n=1000,$ $g\equiv B\equiv 1$. We display the relative errors in $\mathbb{L}^2$ norms, (defined by $||\phi - \hat{\phi}||_{\mathbb{L}^2} /||\phi||_{\mathbb{L}^2}$),  and their empirical variances.
In Table \ref{tab:num2}, for the case $g(x)=x$ and $B(x)=x^2,$ we give some results on standard errors for various values of $n,$ and compare them to $n^{-1/5},$ which is the order of magnitude of the expected final error on $BN,$ since with a Gaussian kernel we have $s=1$ in Proposition \ref{rate}. We see that our numerical results are in line with the theoretical estimates: indeed, the error on $H$ is roughly twice as large as $n^{-1/5}.$\\

\begin{tabular}{cccc}\label{tab:num} 
Error on $N$: average & Variance& Error  on $\f{\p}{\p x} (gN):$ average & Variance
\\
\hline
0.088 & 0.15 & 0.51 & 0.28 \\ 
 Error on $BN$: average & Variance & Average $\hat h$ & Average $\tilde h$  \\
 \hline
0.39 &0.29 &0.12&0.40 
\end{tabular}

\vspace{0.5cm}

\begin{tabular}{ccccccc}\label{tab:num2}
 $n$ & $n^{-\f{1}{5}}$ & $\hat{h}$ & $\tilde{h}$ & error on $N$ & error on $D$ & error on $H$ \\
\hline
$10^3$ & 0.25 & 0.1 & 0.5 & 0.06 & 0.68 & 0.42 \\
$5.10^3$ & 0.18 & 0.07 & 0.3 & 0.03 & 0.45 & 0.28 \\
$10^4$ & 0.16 & 0.08 & 0.3 & 0.035 & 0.46 & 0.29 \\
$5.10^4$ & 0.11 & 0.04 & 0.2 & 0.014 & 0.31 & 0.19 
\end{tabular}

\section{Proofs}\label{proofs-section}
In Section~\ref{proof-main-section}, we first give the proofs of the main results of Section~\ref{proposed-section}. This allows us, in Section \ref{proof-main-main-section}, to prove the results of  Section~\ref{oracle-section}, which require the collection of all the results of Section~\ref{proposed-section}, \emph{i.e.} the oracle-type inequalities on the one hand and a numerical analysis result on the other hand. This illustrates the subject of our paper that lies at the frontier between these fields.  Finally, we state and prove the technical lemmas used in Section~\ref{proof-main-section}. These technical tools are concerned with probabilistic results, namely concentration and Rosenthal-type inequalities that are often the main bricks to establish oracle inequalities, and also the boundedness of ${\cal L}^{-1}_k$. 
In the sequel, the notation $\square_{\theta_1,\theta_2,\ldots}$ denotes a generic positive constant  depending on $\theta_1,\theta_2,\ldots$ (the notation $\square$ simply denotes a generic positive
absolute constant).  It means that the values of $\square_{\theta_1,\theta_2,\ldots}$
may change from line to line.
\subsection{Proofs of the main results of Section \ref{proposed-section}}\label{proof-main-section}
\subsubsection*{Proof of Proposition~\ref{estN}}
For any $h^*\in\HH$, we have:
\begin{eqnarray*}
\|\hat N-N\|_2&\leq&\|\hat N_{\hat h}-N_{\hat h,h^*}\|_2+\|N_{\hat h,h^*}-\hat N_{h^*}\|_2+\|\hat N_{h^*}-N\|_2\\
&\leq &A_1+A_2+A_3,
\end{eqnarray*}
with
\[ A_1:=\|\hat N_{\hat h}-N_{\hat h,h^*}\|_2\leq A(h^*)+\frac{\chi}{\sqrt{n\hat h}}\|K\|_2,\]
\[ A_2:=\|N_{\hat h,h^*}-\hat N_{h^*}\|_2\leq A(\hat h)+\frac{\chi}{\sqrt{nh^*}}\|K\|_2\]
and 
\[ A_3:=\|\hat N_{h^*}-N\|_2.\]
We obtain
\begin{eqnarray}\label{Risk}
\|\hat N-N\|_2&\leq&A(h^*)+\frac{\chi}{\sqrt{n\hat h}}\|K\|_2+ A(\hat h)+\frac{\chi}{\sqrt{nh^*}}\|K\|_2+\|\hat N_{h^*}-N\|_2\nonumber\\
&\leq&2A(h^*)+\frac{2\chi}{\sqrt{nh^*}}\|K\|_2+\|\hat N_{h^*}-N\|_2.
\end{eqnarray}
Since we have
\begin{eqnarray}\label{Ah*}
A(h^*)&=&\sup_{h'\in\HH}\big\{\|\hat N_{h^*,h'}-\hat N_{h'}\|_2-\frac{\chi}{\sqrt{nh'}}\|K\|_2\big\}_+\nonumber\\
&\leq&\sup_{h'\in\HH}\big\{\big\{\|\hat N_{h^*,h'}-\E[\hat N_{h^*,h'}]-\big(\hat N_{h'}-\E[\hat N_{h'}]\big)\|_2-\frac{\chi}{\sqrt{nh'}}\|K\|_2\big\}_+\nonumber\\
&&\hspace{1cm}+\|\E[\hat N_{h^*,h'}]-\E[\hat N_{h'}]\|_2\big\}
\end{eqnarray}
and for any $x$ and any $h'$
\begin{eqnarray*}
 \E\big[\hat N_{h^*,h'}(x))-\E(\hat N_{h'}(x)\big]&=&\int (K_{h^*}\star K_{h'})(x-u)N(u)du-\int K_{h'}(x-v)N(v)dv\\
&=&\int\int K_{h^*}(x-u-t)K_{h'}(t)N(u)dt du-\int K_{h'}(x-v)N(v)dv\\
&=&\int\int K_{h^*}(v-u)K_{h'}(x-v)N(u)dudv-\int K_{h'}(x-v)N(v)dv\\
&=&\int K_{h'}(x-v)\big(\int K_{h^*}(v-u)N(u)du-N(v)\big)dv,
\end{eqnarray*}
we derive
\begin{eqnarray}\label{biais}
\|\E(\hat N_{h^*,h'})-\E(\hat N_{h'})\|_2&\leq&\|K\|_1\|E_{h^*}\|_2,
\end{eqnarray}
where 
\[E_{h^*}(x):=(K_{h^*}\star N)(x)-N(x),\;\;x\in\R_+\]
represents the approximation term. Combining (\ref{Risk}), (\ref{Ah*}) and (\ref{biais}) entails
\begin{eqnarray*}
\|\hat N-N\|_2&\leq&\|\hat N_{h^*}-N\|_2+2\|K\|_1\|E_{h^*}\|_2+\frac{2\chi}{\sqrt{nh^*}}\|K\|_2+2\zeta_n,
\end{eqnarray*}
with 
\begin{eqnarray*}
\zeta_n&:=&\sup_{h'\in\HH}\big\{\|\hat N_{h^*,h'}-\E[\hat N_{h^*,h'}]-\big(\hat N_{h'}-\E[\hat N_{h'}]\big)\|_2-\frac{\chi}{\sqrt{nh'}}\|K\|_2\big\}_+\\
&=&\sup_{h'\in\HH}\big\{\|K_{h^*}\star\big(\hat N_{h'}-\E[\hat N_{h'}]\big)-(\hat N_{h'}-\E[\hat N_{h'}])\|_2-\frac{(1+\e)(1+\|K\|_1)}{\sqrt{nh'}}\|K\|_2\big\}_+\\
&\leq &(1+\|K\|_1)\sup_{h'\in\HH}\big\{\|\hat N_{h'}-\E[\hat N_{h'}]\|_2-\frac{(1+\e)}{\sqrt{nh'}}\|K\|_2\big\}_+.
\end{eqnarray*}
Hence
\[
\E\big[\|\hat N-N\|_2^{2q}\big]\leq\square_q\big(\E\big[\|\hat N_{h^*}-N\|_2^{2q}\big]+\norme{K}_1^{2q}\|E_{h^*}\|_2^{2q}+\chi^{2q}\frac{\|K\|_2^{2q}}{(nh^*)^q}+(1+\norme{K}_1)^{2q}\E[\xi_n^{2q}]\big),
\]
where
\[\xi_n=\sup_{h'\in\HH}\big\{\|\hat N_{h'}-\E(\hat N_{h'})\|_2-\frac{(1+\e)}{\sqrt{nh'}}\|K\|_2\big\}_+.\]
Now, we have:
\begin{eqnarray*}
\E\big[\|\hat N_{h^*}-N\|_2^{2q}\big]&\leq&2^{2q-1}\big( \E\big[\|\hat N_{h^*}-\E[\hat N_{h^*}]\|_2^{2q}\big]+\|\E[\hat N_{h^*})-N\|_2^{2q}\big]\\
&\leq&2^{2q-1}\big( \E\big[\|\hat N_{h^*}-\E[\hat N_{h^*}]\|_2^{2q}\big]+\|E_{h^*}\|_2^{2q}\big).
\end{eqnarray*}
Then, by setting
\[Kc_{h^*}(X_i,x):= K_{h^*}(x-X_i)-\E(K_{h^*}(x-X_1)),\]
we obtain
\begin{eqnarray*}
 \E\big[\|\hat N_{h^*}-\E[\hat N_{h^*}]\|_2^{2q}\big]&=&\E\Big[\Big(\int \Big(\frac{1}{n}\sum_{i=1}^n Kc_{h^*}(X_i,x)\Big)^2dx\Big)^q\Big]\\
&\leq&\frac{2^{q-1}}{n^{2q}}\Big(\E\Big[\Big(\sum_{i=1}^n\int Kc_{h^*}^2(X_i,x)dx\Big)^q\Big]\big.\\
&&\hspace{1cm}+\big.\E\Big[\Big|\sum_{1\leq i,j\leq n\ i\not=j}\int Kc_{h^*}(X_i,x)Kc_{h^*}(X_j,x)dx\Big|^q\Big]\Big).
\end{eqnarray*}
Since
\begin{eqnarray*}
\int Kc_{h^*}^2(X_i,x)dx&=&\int\Big(K_{h^*}(x-X_i)-\E\big[K_{h^*}(x-X_1)\big]\Big)^2dx\\
&\leq&2\big(\int K_{h^*}^2(x-X_i)dx+\int\Big(\E\big[K_{h^*}(x-X_1)\big]\Big)^2dx\big)\\
&\leq&2\big(\|K_{h^*}\|_2^2+\int\E\big[K_{h^*}^2(x-X_1)\big]dx\big)\\
&\leq&4\|K_{h^*}\|_2^2=\frac{4}{h^*}\|K\|_2^2,
\end{eqnarray*}
the first term can be bounded as follows
\[\E\big[\big(\int \sum_{i=1}^n Kc_{h^*}^2(X_i,x)dx\big)^q\big]\leq\big(\frac{4n}{h^*}\|K\|_2^2\big)^q.\]
For the second term, we apply Theorem~8.1.6 of de la Pe{\~n}a and Gin{\'e} (1999) (with $2q\geq2$) combined with the Cauchy-Schwarz inequality:
\begin{align*}
& \E\Big[\Big|\sum_{1\leq i,j\leq n\ i\not=j}\int Kc_{h^*}(X_i,x) Kc_{h^*}(X_j,x)dx\Big|^q\Big]\\
\leq\; & \Big(\E\Big[\big| \sum_{1\leq i,j\leq n\ i\not=j}\int Kc_{h^*}(X_i,x)Kc_{h^*}(X_j,x)dx\big|^{2q}\Big]\Big)^{\frac{1}{2}}\\
\leq\;  & \square_qn^{q}\Big(\E\big[\big| \int Kc_{h^*}(X_1,x) Kc_{h^*}(X_2,x)dx\big|^{2q}\big]\Big)^{\frac{1}{2}}\\
\leq\; & \square_qn^{q}\Big(\E\big[\big| \int Kc_{h^*}^2(X_1,x)dx\big|^{2q}\big]\Big)^{\frac{1}{2}}
\leq\; \square_q\big(\frac{4n}{h^*}\|K\|_2^2\big)^q.
\end{align*}
It remains to deal with the term $\E(\xi_n^{2q})$.
By Lemma~\ref{concentration} below, we obtain 
\[\E[\xi_n^{2q}]\leq \square_{q, \eta, \delta \norme{K}_2, \norme{K}_1, \norme{N}_\infty} n^{-q}\]
and the conclusion follows.
\subsubsection*{Proof of Proposition \ref{estD}}
The proof is similar to the previous one and we avoid most of the computations for simplicity. For any $h_0\in~\tilde\HH$,
\begin{eqnarray*}
\|\hat D_{\tilde h}-D\|_2&\leq&\|\hat D_{\tilde h}-\hat D_{\tilde h,h_0}\|_2+\|\hat D_{\tilde h,h_0}-\hat D_{h_0}\|_2
+\|\hat D_{h_0}-D\|_2\\
&\leq&\tilde A_1+\tilde A_2+\tilde A_3,
\end{eqnarray*}
with
\[\tilde A_1:=\|\hat D_{\tilde h}-\hat D_{\tilde h,h_0}\|_2\leq \tilde A(h_0)+\frac{\tilde\chi}{\sqrt{n\tilde h^3}}\|g\|_\infty\|K'\|_2,\]
\[\tilde A_2:=\|\hat D_{\tilde h,h_0}-\hat D_{h_0}\|_2\leq \tilde A(\tilde h)+\frac{\tilde\chi}{\sqrt{nh_0^3}}\|g\|_\infty\|K'\|_2\]
and
\[\tilde A_3:=\|\hat D_{h_0}-D\|_2.\]
Then,
\[\|\hat D_{\tilde h}-D\|_2\leq 2\tilde A(h_0)+\frac{2\tilde\chi}{\sqrt{nh_0^3}}\|g\|_\infty\|K'\|_2
+\|\hat D_{h_0}-D\|_2.\]
To study $\tilde A(h_0)$, we first evaluate 
\begin{eqnarray*}
\E[\hat D_{h_1,h_2}(x)]-\E[\hat D_{h_2}(x)].&=&(K_{h_1}\star K_{h_2}\star (gN)')(x)- (K_{h_2}\star (gN)')(x)\\
&=&\int D(t)(K_{h_1}\star K_{h_2})(x-t)dt-\int D(t)K_{h_2}(x-t)dt\\
&=&\int D(t)\int K_{h_1}(x-t-u)K_{h_2}(u)dudt-\int D(t)K_{h_2}(x-t)dt\\
&=&\int D(t)\int K_{h_1}(v-t)K_{h_2}(x-v)dvdt-\int D(v)K_{h_2}(x-v)dv\\
&=&\int K_{h_2}(x-v)\Big(\int D(t)K_{h_1}(v-t)dt-D(v)\Big)dv\\
&=&(K_{h_2}\star \tilde E_{h_1})(x),
\end{eqnarray*}
where we set, for any real number $x$
\begin{eqnarray}\label{biaisD}
\tilde E_{h_1}(x)&:=&\int D(t)K_{h_1}(x-t)dt-D(x)\nonumber\\
&=&(K_{h_1}\star D)(x)-D(x).
\end{eqnarray}
It follows that
\begin{eqnarray}\label{tildeA}
\tilde A(h_0)&=&\sup_{h\in\tilde\HH}\big\{\|\hat D_{h_0,h}-\hat D_{h}\|_2-\frac{\tilde\chi}{\sqrt{nh^3}}\|g\|_\infty\|K'\|_2\big\}_+\nonumber\\
&\leq&\sup_{h\in\tilde\HH}\big\{\big\{\|\hat D_{h_0,h}-\E[\hat D_{h_0,h}]-\big(\hat D_{h}-\E[\hat D_{h}]\big)\|_2-\frac{\tilde\chi}{\sqrt{nh^3}}\|g\|_\infty\|K'\|_2\big\}_+\big.\nonumber\\
&&\hspace{1cm}+\big.\|\E[\hat D_{h_0,h}]-\E[\hat D_{h}]\|_2\big\}\nonumber\\
&\leq &\sup_{h\in\tilde\HH}\big\{\|\hat D_{h_0,h}-\E[\hat D_{h_0,h}]-(\hat D_{h}-\E[\hat D_{h}])\|_2-\frac{\tilde\chi}{\sqrt{nh^3}}\|g\|_\infty\|K'\|_2\big\}_++\|K\|_1\|\tilde E_{h_0}\|_2\nonumber\\
&\leq & (1+\|K\|_1)\sup_{h\in\tilde\HH}\big\{\|\hat D_{h}-\E[\hat D_{h}]\|_2-\frac{(1+\tilde\e)}{\sqrt{nh^3}}\|g\|_\infty\|K'\|_2\big\}_++\|K\|_1\|\tilde E_{h_0}\|_2,
\end{eqnarray}
In order to obtain the last line, we use the following chain of arguments:
\begin{eqnarray*}
\hat D_{h_0,h}(x)&=&\frac{1}{n}\sum_{i=1}^ng(X_i)\int K_h'(x-X_i-t)K_{h_0}(t)dt\\
&=&\int K_{h_0}(t)\Big(\frac{1}{n}\sum_{i=1}^ng(X_i) K_h'(x-X_i-t)\Big)dt
\end{eqnarray*}
and
\begin{eqnarray*}
\E\big[\hat D_{h_0,h}(x)\big]&=&\int K_{h_0}(t) \Big(\int g(u)K_h'(x-u-t)N(u)du \Big)dt,
\end{eqnarray*}
therefore
\[
\hat D_{h_0,h}(x)-\E\big[\hat D_{h_0,h}(x)\big]=\int K_{h_0}(t)G(x-t)dt=K_{h_0}\star G(x),
\]
with
\begin{eqnarray*}
G(x)&=&\frac{1}{n}\sum_{i=1}^ng(X_i) K_h'(x-X_i)-\int g(u)K_h'(x-u)N(u)du\\
&=&\hat D_h(x)-\E\big[\hat D_h(x)\big].
\end{eqnarray*}
Therefore
\begin{eqnarray*}
\|\hat D_{h_0,h}-\E[\hat D_{h_0,h}]\|_2&\leq& \|K_{h_0}\|_1\|G\|_2\\
&\leq&\|K\|_1\|\hat D_h-\E[\hat D_h]\|_2,
\end{eqnarray*}
which justifies (\ref{tildeA}). In the same way as in the proof of Proposition~\ref{estN}, we can establish the following:
\[\E\big[\|\tilde E_{h_0}\|_2^{2q}\big]=\E\big[\|\hat D_{h_0}-D\|_2^{2q}\big]\leq \square_q \Big(\|\tilde E_{h_0}\|_2^{2q}+\Big(\frac{\norme{g}_\infty\|K'\|_2}{\sqrt{nh_0^3}}\Big)^{2q}\Big).\]
Finally,  we successively apply (\ref{biaisD}), (\ref{tildeA}) and Lemma~\ref{concentration} in order to conclude the proof.
\subsubsection*{Proof of Proposition \ref{stabiliteL2}}
We use the notation and definitions of Section \ref{sec:inversion}. We have
$$ \|{\mathcal L}^{-1}_k(\varphi)-{\mathcal L}^{-1}(\varphi)\|_{2,T}^2
=
\sum\limits_{i=0}^{k-1} \int\limits_{x_{i,k}}^{x_{i+1,k}} \big(H_{i,k} - {\cal L}^{-1} (\varphi) (x)\big)^2 dx:=\sum\limits_{i=0}^{k-1} L_{i,k}.$$
We prove by induction that for all $i,$ one has $L_{i,k} \leq C^2 \f{T^2}{k^2} \|\varphi\|_{ \mathcal{W}^{1}}^2$. The result follows by summation over $i.$ We first prove the two following estimates:
\begin{eqnarray}
& \int\limits_{x_{i,k}}^{x_{i+1,k}} \big(\varphi_{i,k} - \varphi(x)\big)^2 dx \leq \f{T^2}{4\pi^2 k^2}\|\varphi\|^2_{{\cal W}^1}, \\ \label{ass:delta} 
& |\varphi_{i+1,k} - \varphi_{i,k}|^2 \leq \f{T}{k}  \|\varphi\|^2_{{\cal W}^1}. \label{ass:diff}
\end{eqnarray}
By definition,  $\varphi_{i,k}$ is the average of the function $\varphi$ on the interval $[x_{i,k},\;x_{i+1,k}]$ of size $\f{T}{k}.$ Thus \eqref{ass:delta} is simply Wirtinger inequality applied to $\varphi\in{\cal W}^1$ on the interval $[x_{i,k},\;x_{i+1,k}]. For $\eqref{ass:diff}, we use the Cauchy-Schwarz inequality:
\begin{align*}
|\varphi_{i+1,k} -\varphi_{i,k}|^2&=\f{k^2}{T^2} \biggl(\int\limits_{x_{i,k}}^{x_{i+1,k}} \big(\varphi(x+\f{T}{k})-\varphi(x) \bigr) dx\biggr)^2= \f{k^2}{T^2} \bigl(\int\limits_{x_{i,k}}^{x_{i+1,k}} \int\limits_x^{x+\f{T}{k}} \varphi'(z)dz\,dx\bigr)^2 
\\
&\leq \f{k^2}{T^2} \biggl(\int\limits_{x_{i,k}}^{x_{i+1,k}} \sqrt{\f{T}{k}} \| \varphi\|_{{\cal W}^1} dx\biggr)^2 = 
\f{T}{k} \| \varphi\|_{{\cal W}^1}^2.
\end{align*}
We are ready to prove by induction the two following inequalities:
\begin{eqnarray}
 & L_{i,k} \leq C_1^2 \f{T^2}{k^2} \|\varphi\|_{{\cal W}^1}^2,  \label{eq 1}\\ 
 & |H_{i+1,k} -H_{i,k}|^2 \leq C_2^2 \f{T}{k}  \|\varphi\|_{{\cal W}^1}^2. \label{eq 2}
\end{eqnarray}
for two constants $C_1$ and $C_2$ specified later on. First, for $i=0,$ we have
$$L_{0,k}=\int\limits_0^{\f{T}{k}}|H_{0,k} (\varphi) - {\cal L}^{-1} (\varphi) (x)|^2 dx=\int\limits_0^{\f{T}{k}}\big|\f{1}{3} \varphi_{0,k} - {\cal L}^{-1} (\varphi) (x)\big|^2 dx.$$
We recall  (see Proposition A.1. of \cite{DPZ})  that ${\cal L}^{-1} (\varphi) (x) =\sum\limits_{n=1}^\infty 2^{-2n} \varphi(2^{-n}x),$
and we use the fact that $\f{1}{3} =\sum\limits_{n=1}^\infty 2^{-2n}$ and for $a,b>0$, $ab\leq\frac{1}{2}(a^2+b^2)$ in order to write
\begin{align*}
L_{0,k}&=\int\limits_0^{\f{T}{k}}\big|\sum\limits_{n=1}^\infty 2^{-2n} \big(\varphi_{0,k} - \varphi(2^{-n} x)\big)\big|^2 dx
\leq  \sum\limits_{n,n'=1}^\infty 2^{-2n-2n'} \int\limits_0^{\f{T}{k}}| \varphi_{0,k} - \varphi(2^{-n} x) |^2 dx
\\
&\leq \f{1}{3} \sum\limits_{n=1}^\infty 2^{-n} \int\limits_0^{2^{-n}\f{T}{k}} | \varphi_{0,k} - \varphi(y) |^2 dy \leq \f{1}{3} \f{T^2}{4\pi^2 k^2}\|\varphi \|_{{\cal W}^1}^2.
\end{align*}
This proves the first induction assumption for $i=0,$ and
$$|H_{1,k} - H_{0,k}|^2=\big|\f{1}{7} (\varphi_{1,k} - \varphi_{0,k})\big|^2 \leq \f{1}{7^2} \f{T}{k} \|\varphi\|_{{\cal W}^1}^2,$$
proves the second one. Let us now suppose that the two induction assumptions are true for all $j\leq i-1,$ and take $i\geq 1.$ Let us first evaluate
$$L_{i,k}=\int\limits_{x_{i,k}}^{x_{i+1,k}} \big(H_{i,k} - {\cal L}^{-1}(\varphi)(x)\big)^2 dx=
\f{1}{16} \int\limits_{x_{i,k}}^{x_{i+1,k}} \big(H_{\f{i}{2},k} +\varphi_{\f{i}{2},k} - {\cal L}^{-1}(\varphi)(\f{x}{2}) 
 - \varphi(\f{x}{2})\big)^2 dx.$$
We distinguish the case when $i$ is even and when $i$ is odd. Let $i$ be even: then, by definition
$$L_{i,k} \leq \f{1}{8} 
\int\limits_{x_{i,k}}^{x_{i+1,k}} \big(H_{\f{i}{2},k}  - {\cal L}^{-1}(\varphi)(\f{x}{2})\big)^2 dx
+ \f{1}{8} \int\limits_{x_{i,k}}^{x_{i+1,k}} \big(\varphi_{\f{i}{2},k}- \varphi(\f{x}{2})\big)^2 dx
 \leq \f{1}{4}(C_1^2+\f{1}{4\pi^2})\f{T^2}{k^2} \|\varphi\|^2_{{\cal W}^1}$$
by the induction assumption and Assertion \ref{ass:delta} on $\varphi$ for $j=\f{i}{2}.$ If $i$ is odd, we write by definition
$$L_{i,k}=
\f{1}{16} \int\limits_{x_{i,k}}^{x_{i+1,k}} \biggl(H_{\f{i-1}{2},k} +\varphi_{\f{i-1}{2},k} - {\cal L}^{-1}(\varphi)(\f{x}{2}) 
 - \varphi(\f{x}{2})
+\f{1}{2}(H_{\f{i+1}{2},k} -  H_{\f{i-1}{2},k}) +\f{1}{2}(\varphi_{\f{i+1}{2},k} -\varphi_{\f{i-1}{2},k}) 
\biggr)^2 dx.$$
Hence, re-organizing terms, we can write
\begin{align*} L_{i,k}&\leq
\f{1}{2} \int\limits_{x_{\f{i-1}{2},k}}^{x_{\f{i-1}{2}+1,k}} \biggl(H_{\f{i-1}{2},k} 
- {\cal L}^{-1}(\varphi)({y})\biggr)^2 dy
+\f{1}{2}\int\limits_{x_{\f{i-1}{2},k}}^{x_{\f{i-1}{2}+1,k}} 
\biggl(\varphi_{\f{i-1}{2},k}   - \varphi({y})\biggr)^2 dy \\
&+\f{1}{16}\f{T}{k} (H_{\f{i+1}{2},k} -  H_{\f{i-1}{2},k})^2 +\f{1}{16}\f{T}{k}(\varphi_{\f{i+1}{2},k} -\varphi_{\f{i-1}{2},k})^2 
\end{align*}
Putting together the four inequalities above (the estimates for $\varphi$ and the induction assumptions), we obtain
$$L_{i,k}\leq \f{T^2}{k^2} \|\varphi\|^2_{{\cal W}^1} \biggl(\f{C_1^2}{2} + \f{1}{8\pi^2} + \f{C_2^2}{16} + \f{1}{16}\biggr)$$
and \eqref{eq 1} is proved. It remains to establishe \eqref{eq 2}. Let us write it for $i$ even (the case of an odd $i$ is similar):
$$|H_{i+1,k} - H_{i,k}|^2=\f{1}{16} |H_{\f{i+1}{2}} - H_{\f{i}{2}} + \varphi_{\f{i+1}{2}} - \varphi_{\f{i}{2}}|^2
= \f{1}{32} |H_{\f{i}{2} +1} - H_{\f{i}{2}} + \varphi_{\f{i}{2}+1} - \varphi_{\f{i}{2}}|^2.$$
Hence, as previously, we obtain
 $$|H_{i+1,k} - H_{i,k}|^2 \leq \f{1}{16} \f{T}{k}\|\varphi\|_{{\cal W}^1}^2(C_2^2+1).$$
To complete the proof, we remark that  $C_2^2=\f{1}{15}$ and $C_1^2=\f{1}{4\pi^2} +\f{1}{8}(1+\f{1}{15}) <\f{1}{6}$ are suitable. It is consequently sufficient to take $C=C_1$.
\subsection{Proof of Theorem \ref{oraclesurH} and Proposition \ref{rate}}\label{proof-main-main-section}
\subsubsection*{Proof of Theorem \ref{oraclesurH}}
It is easy to see that
\begin{eqnarray*}
\|\widehat H-H\|_{2,T}&=&\|{\mathcal L}^{-1}_k(\hat\kappa\hat D+\hat\lambda_n\hat N)-{\mathcal L}^{-1}({\mathcal L}(BN))\|_{2,T}\\
&\leq&\|{\mathcal L}^{-1}_k(\hat\kappa\hat D+\hat\lambda_n\hat N)-{\mathcal L}^{-1}_k({\mathcal L}(BN))\|_{2,T}\\
&&\hspace{1.5cm}+\|{\mathcal L}^{-1}_k({\mathcal L}(BN))-{\mathcal L}^{-1}({\mathcal L}(BN))\|_{2,T}\\
&\leq&\|{\mathcal L}^{-1}_k(\hat\kappa \hat D+\hat\lambda_n\hat N - (\kappa D+\lambda N))\|_{2,T}
\\
&&\hspace{1.5cm}+\f{1}{3}\f{T}{\sqrt{k}}\|{\cal L} (BN)\|_{{\cal W}^1},
\end{eqnarray*}
thanks to Proposition \ref{stabiliteL2}. Note that ${\mathcal L}(BN)=\kappa (gN)'+\lambda N$ so that we can write 
$$\|{\cal L} (BN)\|_{{\cal W}^1} \leq C (\|N\|_{{\cal W}^1} +\|g N\|_{{\cal W}^2}).$$ 
We obtain, thanks to Lemma \ref{lk-1} that gives the boundedness of the operator ${\cal L}_k^{-1}:$
\begin{eqnarray*}
\|\widehat H-H\|_{2,T}&\leq &\square\big(\norme{\hat\kappa_n\hat D-\kappa D}_{2,T}+\norme{\hat\lambda_n\hat N-\lambda N}_{2,T}
+(\|N\|_{{\cal W}^1} +\|g N\|_{{\cal W}^2})\f{T}{\sqrt{k}}\big)\\
&\leq & \square \big(|\hat\kappa_n|\|\hat D-D\|_2+|\hat\lambda_n|\|\hat N-N\|_2+|\hat\kappa_n-\kappa|\|D\|_2+|\hat\lambda_n-\lambda|\|N\|_2\big.\\
&&\hspace{2cm}\big.+  (\|N\|_{{\cal W}^1} +\|g N\|_{{\cal W}^2})\f{T}{\sqrt{k}}
\big)\\
&\leq & \square \big(|\hat\lambda_n||\hat\rho_n|\|\hat D-D\|_2+|\hat\lambda_n|\big(\|\hat N-N\|_2+|\hat\rho_n-\rho_g(N)|\|D\|_2\big)\big.\\
&&\hspace{1cm}\big.+\big(\|N\|_2+|\rho_g(N)|\|D\|_2\big)|\hat\lambda_n-\lambda|+(\|N\|_{{\cal W}^1} +\|g N\|_{{\cal W}^2})\f{T}{\sqrt{k}}\big).
\end{eqnarray*}
Taking expectation and using Cauchy-Schwarz inequality,  we obtain for any $q\geq 1$,
{\small
\begin{eqnarray*}
\E[\|\widehat H-H\|_{2,T}^q]&\leq & \square_q\Big[\big(\E[\hat\lambda_n^{2q}]\big)^{1/2}\Big\{\big(\E[{\hat\rho_n}^{4q}]\big)^{1/4}\big(\E[\|\hat{D}-D\|_2^{4q}]\big)^{1/4}+\big(\E[\|\hat{N}-N\|_2^{2q}]\big)^{1/2} \\
&&\big.+\|D\|_{2}^{q}\big(\E[|\hat\rho_n-\rho_g(N)|^{2q}]\big)^{1/2}\big\}\\
&&\big. +\big(\|N\|_{2}+\rho_g(N) \norme{D}_{2}\big) ^q\E[|\hat\lambda_n-\lambda|^q]+ \big((\|N\|_{{\cal W}^1} +\|g N\|_{{\cal W}^2})\f{T}{\sqrt{k}}\big)^q\big].
\end{eqnarray*}
}
Now, Lemma~\ref{hatrho} gives the behaviour of $\E[|\hat\rho_n-\rho_g(N)|^{2q}]$. In particular, we obtain
\[\E[{\hat\rho_n}^{4q}] \leq \square_{q,g,N,c}.\] 
We finally apply successively Propositions \ref{estN} and \ref{estD} to obtain the proof of Theorem \ref{oraclesurH}.
\subsubsection*{Proof of Proposition \ref{rate}}
\label{proof-section-B}
We have already proved \eqref{0}. It remains to prove \eqref{1}.
We introduce the event
$$\Omega_n=\{2\hat{N}(x)\geq m \mbox{ for any } x\in[a,b]\}.$$
Then, for $n$ larger that $Q^2$,
\begin{eqnarray*}
& &\E\Big[\Big(\int_a^b \big(\tilde{B}(x)-B(x)\big)^2 dx\Big)^{\frac{q}{2}}\Big]\\
&=&\E\Big[\Big(\int_a^b \big(\tilde{B}(x)-B(x)\big)^2 dx \times {\bf 1}_{\Omega_n}\Big)^{\frac{q}{2}}\Big]+\E\big[\big(\int_a^b \big(\tilde{B}(x)-B(x)\big)^2 dx \times {\bf 1}_{\Omega_n^c}\big)^{\frac{q}{2}}\big]\\
&\leq&\E\Big[\Big(\int_a^b \big({\hat B}(x)-B(x)\big)^2 dx\times {\bf1}_{\Omega_n}\Big)^{\frac{q}{2}}\Big]+\big(2(b-a)(n+Q^2)\big)^{\frac{q}{2}}\PP(\Omega_n^c)\\
&\leq&\E\Big[\Big(\int_a^b \Big(\frac{\hat H}{\hat N}-\frac{H}{N}\Big)^2\times {\bf 1}_{\Omega_n}\Big)^{\frac{q}{2}}\Big] +\big(4n(b-a)\big)^{\frac{q}{2}}\PP(\Omega_n^c)\\
&\leq& \E\Big[\Big(\int_a^b \Big(\frac{\hat H N- \hat N H}{\hat N N}\Big)^2\times {\bf 1}_{\Omega_n}\Big)^{\frac{q}{2}} \Big]+\big(4n(b-a)\big)^{\frac{q}{2}}\PP(\Omega_n^c)\\
&\leq & \square_{q,m,M,Q} \big( \E\big[ \norme{\hat H -H}^q_{2}\big]+ \E\big[ \norme{\hat N -N}^q_2\big] \big)+\big(4n(b-a)\big)^{\frac{q}{2}}\PP(\Omega_n^c).
\end{eqnarray*}
The first term of the right hand side is handled by \eqref{0} and Proposition~\ref{estN}. The second term is handled by Lemma~\ref{bar} that establishes that $\PP(\Omega_n^c)=O(n^{-q}).$
\subsection{Technical lemmas}\label{technique}
\subsubsection*{Concentration inequalities}
We first state the following concentration result. Note that a more general version of this result can be found in \cite{LepsP}. We nevertheless give a proof for the sake of completeness.
\begin{lemma}\label{concentration}
We have the following estimates
\begin{itemize}
\item Assume that $\norme{K}_2, \norme{K}_1, \norme{g}_\infty$ and $\norme{N}_\infty$ are finite. For every $q>0$,
introduce the grid $\mathcal{H}\subset\{1,1/2,..., 1/D_{\max}\}$ and $D_{\max}=\delta n$ for some $\delta>0$. Then, for every $\eta>0$
\[\E\big[\sup_{h\in\HH}\big\{\|\hat N_{h}-\E[\hat N_{h}]\|_2-\frac{(1+\eta)}{\sqrt{nh}}\|K\|_2\big\}_+^{2q}\big]\leq \square_{q, \eta, \delta \|K\|_2, \|K\|_1, \|N\|_\infty} n^{-q}.\]
\item Assume that $\|K'\|_2, \|K'\|_1, \|g\|_\infty$ and $\|N\|_\infty$ are finite. For every $q>0$,
introduce the grid $\mathcal{\tilde H}\subset\{1,1/2,..., 1/\tilde D_{\max}\}$ and $\tilde D_{\max}=  \sqrt{\tilde\delta n}$ for some $\tilde\delta>0$.  Then for every $\eta>0$
\[\E\big[\sup_{h\in\tilde\HH}\big\{\|\hat D_{h}-\E[\hat D_{h}]\|_2-\frac{(1+\eta)}{\sqrt{nh^{3}}}\|g\|_\infty \|K'\|_2\big\}_+^{2q}\big]\leq \square_{q, \eta , \tilde\delta, \|K'\|_2, \|K'\|_1, \|g\|_\infty,\|N\|_\infty} n^{-q}.\]
\end{itemize}
\end{lemma}
\begin{proof}
Let $X$ be a real random variable and let us consider the random process 
\[\forall\, t \in \mathbb{R},~~w(t,X)=\varphi(X)\Psi(t-X),\]
where $\varphi$ and $\Psi$ are measurable real-valued functions.
Let $X_1,...,X_n$ be $n$ independent and identically distributed random variables with the same distribution as $X$ and let us consider the process
\[\forall\, t \in \R,~~\xi_{\varphi,\Psi}(t)=\sum_{i=1}^n \big(w(t,X_i)-\E\big[w(t,X)\big]\big).\]
First, let us study the behavior of 
\[\norme{\xi_{\varphi,\Psi}}_2= \Big(\int \xi_{\varphi,\Psi}^2(t) dt\Big)^{1/2}.\]
If $\mathcal{B}$ denotes the unit ball in $\mathbb{L}^2$ and $\mathcal{A}$ is a countable subset of $\mathcal{B}$, we have
\begin{eqnarray*}
\norme{\xi_{\varphi,\Psi}}_2&=&\sup_{a\in\mathcal{B}} \int a(t)  \xi_{\varphi,\Psi}(t) dt \\
&=& \sup_{a\in \mathcal{A}} \int a(t)  \xi_{\varphi,\Psi}(t) dt\\
&=& \sup_{a\in \mathcal{A}} \sum_{i=1}^n \int a(t)  \big(w(t,X_i)-\E\big[w(t,X)\big]\big) dt.
\end{eqnarray*}
Hence one can apply Talagrand's inequality (see the version of Bousquet  in the independent and identically distributed case \cite[p 170]{mas}). For all  $\ve,x >0$, one has:
\[\PP\big( \norme{\xi_{\varphi,\Psi}}_2 \geq (1+\ve) \E[\norme{\xi_{\varphi,\Psi}}_2] +\sqrt{2vx} + c(\ve) bx\big)\leq e^{-x},\]
where $c(\ve)=1/3+\ve^{-1},$
 \[v=n\sup_{a\in \mathcal{A}}  \E\big[ \big(\int a(t)  \big[w(t,X)-\E\big(w(t,X)\big)\big] dt\big)^2\big],\]
and
\[b=\sup_{y\in \R, a\in\mathcal{A}} \int a(t)  \big[w(t,y)-\E\big(w(t,X)\big)\big] dt.\]
We study each term of the right hand term within the expectation.
\begin{itemize}
\item Obviously, one has:
\begin{align*}
\E\big[\|\xi_{\varphi,\Psi}\|_2\big] \leq\; & \Big(\E\big[ \int \Big(\sum_{i=1}^n \big(w(t,X_i)-\E\big[w(t,X)\big]\big)\Big)^2 dt \Big)^{1/2}\\
=\; &\Big( \int \sum_{i=1}^n \E\big[ \big(w(t,X_i)-\E\big[w(t,X)\big]\big)^2 \big]dt\Big)^{1/2} \\
\leq\; &  \Big( n \int \E \big[w(t,X)^2\big] dt \Big)^{1/2}.
\end{align*}
But we easily see that for all $y$, 
\begin{equation}
\label{utile}
\int w^2(t,y)dt \leq \|\varphi\|_\infty^2 \|\Psi\|_2^2,
\end{equation}
hence
\[\E[\norme{\xi_{\varphi,\Psi}}_2] \leq \sqrt{n} \norme{\varphi}_\infty \norme{\Psi}_2.\]
\item 
Since $v$ is a supremum of variance terms, 
\begin{eqnarray*}
v&\leq & n  \sup_{a\in \mathcal{A}}   \E\big[ \big(\int a(t) w(t,X) dt \big)^2 \big] \\
&\leq & n \sup_{a\in \mathcal{A}} \E\big[ \int |w(t,X)| dt \int a^2(t) |w(t,X)| dt\big]
\\
&\leq& n \sup_{y\in \R} \int |w(t,y)| dt \times \sup_{t\in\R} \E[|w(t,X)|]\\
&\leq & n \|\varphi\|_\infty^2 \|N\|_\infty \|\Psi\|_1^2.
\end{eqnarray*}
\item The Cauchy-Schwarz inequality and \eqref{utile} give
\begin{eqnarray*}
b&= & \sup_{y\in\R} \|w(\cdot,y)-\E[w(.,X)]\|_2 \\
&\leq & \sup_{y\in\R}  \|w(\cdot,y)\|_2 + \big(\E\big[\int w(t,X)^2 dt\big]\big)^{1/2} \\
&\leq & 2 \|\varphi\|_\infty \|\Psi\|_2.
\end{eqnarray*}
\end{itemize}
The main point here is that $\sqrt{v}$ may be much smaller than $\E[\|\xi_{\varphi,\Psi}\|_2]$. So, for all  $\ve,x >0$,
\[\PP\big( \|\xi_{\varphi,\Psi}\|_2 \geq (1+\ve)\sqrt{n} \|\varphi\|_\infty \|\Psi\|_2  + \|\varphi\|_\infty \|N\|_\infty^{1/2} \|\Psi\|_1\sqrt{2nx} + 2 c(\ve)  \|\varphi\|_\infty \|\Psi\|_2 x\big)\leq e^{-x}.\]
Now we consider a family $\mathcal{M}$ of possible functions $\varphi$ and $\Psi$.
Let us introduce some strictly positive weights $L_{\varphi,\Psi}$  and let us apply the previous inequality to $x=L_{\varphi,\Psi}+u$ for $u>0$. Hence with probability larger than $1-\sum_{(\varphi, \Psi)\in\mathcal{M}} e^{-L_{\varphi,\Psi}} e^{-u}$, for all $(\varphi,\Psi)\in\mathcal{M}$, one has
\begin{eqnarray}
\label{depart}
\|\xi_{\varphi,\Psi}\|_2 &\leq& (1+\ve)\sqrt{n} \|\varphi \|_\infty \|\Psi \|_2  + \|\varphi \|_\infty \|N\|_\infty^{1/2} \|\Psi \|_1\sqrt{2nL_{\varphi,\Psi}} + 2 c(\ve)  \|\varphi \|_\infty \|\Psi \|_2 L_{\varphi,\Psi} \nonumber\\
&& + \|\varphi\|_\infty \|N\|_\infty^{1/2} \|\Psi\|_1\sqrt{2nu} +  2 c(\ve)  \|\varphi\|_\infty \|\Psi\|_2 u.
\end{eqnarray}
Let 
\[M_{\varphi,\Psi}=(1+\ve)\sqrt{n} \|\varphi\|_\infty \|\Psi\|_2  + \|\varphi\|_\infty \|N\|_\infty^{1/2} \|\Psi\|_1\sqrt{2nL_{\varphi,\Psi}} + 2 c(\ve)  \|\varphi\|_\infty \|\Psi\|_2 L_{\varphi,\Psi}.\]
It is also easy to obtain an upper bound of $R_q$ for any $q \geq 1$ with 
\[R_q=\E\big[\sup_{(\varphi, \Psi)\in \mathcal{M}} \big(\|\xi_{\varphi,\Psi}\|_2-M_{\varphi,\Psi}\big)_+^{2q}\big]=\int_0^{+\infty}\PP\big(\sup_{(\varphi, \Psi)\in \mathcal{M}} \big(\|\xi_{\varphi,\Psi}\|_2-M_{\varphi,\Psi}\big)_+^{2q}\geq x\big) dx.\]
Indeed
\[
R_q\leq \sum_{(\varphi, \Psi)\in \mathcal{M}} \int_0^{+\infty}\PP\big(\big(\norme{\xi_{\varphi,\Psi}}_2-M_{\varphi,\Psi}\big)_+^{2q}\geq x\big) dx.\]
Then, let us take $u$ such that \[x=f(u)^{2q}:= \big(\|\varphi\|_\infty \|N\|_\infty^{1/2} \|\Psi\|_1\sqrt{2nu} +  2 c(\ve)  \|\varphi\|_\infty \|\Psi\|_2 u\big)^{2q},\]
so 
\[dx = 2q\big(f(u)\big)^{2q-1} \big(\sqrt{2n}\|\varphi\|_\infty \|N\|_\infty^{1/2} \|\Psi\|_1\frac{1}{2\sqrt{u}}+2 c(\ve)  \|\varphi\|_\infty \|\Psi\|_2\big) du.\]
Hence
\begin{eqnarray*}
R_q&\leq&  \sum_{(\varphi, \Psi)\in \mathcal{M}} \int_0^{+\infty} e^{-(L_{\varphi,\Psi}+u)} 2q\big(f(u)\big)^{2q-1} \big(\sqrt{2n}\|\varphi\|_\infty \|N\|_\infty^{1/2} \|\Psi\|_1\frac{1}{2\sqrt{u}}+2 c(\ve)  \|\varphi\|_\infty \|\Psi\|_2\big) du \\
&\leq& 2 q \sum_{(\varphi, \Psi)\in \mathcal{M}} e^{- L_{\varphi,\Psi}}\int_0^{+\infty} f(u)^{2q} e^{-u} u^{-1}du,\\
&\leq& \square_{q,\ve} \sum_{(\varphi, \Psi)\in \mathcal{M}} e^{- L_{\varphi,\Psi}} \big[n^q \|N\|_\infty^q \|\varphi\|_\infty^{2q}  \|\Psi\|_1^{2q} \int_0^{+\infty} e^{-u} u^{q-1} du +  \|\varphi\|_\infty^{2q}  \|\Psi\|_2^{2q} \int_0^{+\infty} e^{-u} u^{2q-1} du\big].\\
\end{eqnarray*}
Finally, we have proved that
\begin{equation}
\label{final}
R_q
\leq  \square_{q,\ve}
 \sum_{(\varphi, \Psi)\in \mathcal{M}} e^{- L_{\varphi,\Psi}} \big[n^q \|N\|_\infty^q \|\varphi\|_\infty^{2q}  \|\Psi\|_1^{2q} +  \|\varphi \|_\infty^{2q}  \|\Psi \|_2^{2q} \big].
\end{equation}
Now let us evaluate what this inequality means for each set-up.
\begin{itemize}
\item First, when $\varphi=1/n$ and $\Psi=1/h K(\cdot/h)$, the family $\mathcal{M}$ corresponds to the family  $\mathcal{H}$. In that case $M_{\varphi,\Psi}$ and $L_{\varphi,\Psi}$  will respectively be denoted by $M_h$ and $L_h$. The upper bound given in \eqref{final} becomes
\begin{equation}
\label{finalnoyau}
R_q
\leq  \square_{q,\ve} \sum_{h\in \mathcal{H}} e^{- L_h} \big[n^{-q} \|N\|_\infty^q  \|K\|_1^{2q} +  n^{-2q} h^{-q} \|K\|_2^{2q}\big].
\end{equation}
Now it remains to choose $L_h$. But
\[M_h=(1+\ve)\frac{1}{\sqrt{nh}} \|K\|_2  + \|N\|_\infty^{1/2} \|K\|_1\sqrt{\frac{2L_h}{n}} + 2 c(\ve)  \frac{1}{n \sqrt{h}} \|K\|_2 L_h.\]
Let
$\theta>0$ and let  \[L_h=\frac{\theta^2 \|K\|_2^2}{2\|N\|_\infty \|K\|_1^2 \sqrt{h}}.\]
Obviously the series in \eqref{finalnoyau} is finite and for any $h\in\mathcal{H}$,
since $h\leq 1$, we have:
\[M_h\leq (1+\ve+\theta)\frac{\|K\|_2}{\sqrt{nh}}+\frac{c(\ve) \theta^2 \|K\|_2^3}{\|N\|_\infty \|K\|_1^2}\frac{1}{nh}.\]
Since $D_{\max}=\delta n$, one obtains that
\[M_h\leq \big(1+\ve+\theta+\frac{c(\ve) \theta^2 \|K\|_2^3\sqrt{\delta}}{\|N\|_\infty \|K\|_1^2}\big)\frac{\|K\|_2}{\sqrt{nh}}.\]
It remains to choose $\ve=\eta/2$ and $\theta$ small enough such that $\theta+\frac{c(\ve) \theta^2 \|K\|_2^3\sqrt{\delta}}{\|N\|_\infty \|K\|_1^2}=\eta/2$ to obtain the desired inequality.
\item Secondly, if $\varphi=g/n$ and $\Psi=1/h^2 K'(\cdot/h)$the family $\mathcal{M}$ corresponds to the family $\tilde{\mathcal{H}}$. So,  $M_{\varphi,\Psi}$ and $L_{\varphi,\Psi}$  will  be denoted by  $M'_h$ and $L'_h$ respectively. The upper bound given by \eqref{final} now becomes
\begin{equation}
\label{finalderive}
R_q
\leq \square_{q,\ve} \sum_{h\in \tilde{\mathcal{H}}} e^{- L'_h} \big[n^{-q} h^{-2q} \|N\|_\infty^q\|g\|_\infty^{2q} \|K'\|_1^{2q}  +  n^{-2q} h^{-3q} \|g\|_\infty^{2q} \|K'\|_2^{2q}\big].
\end{equation}
But
\[M'_h=(1+\ve)\frac{1}{\sqrt{n}h^{3/2}} \|g\|_\infty \|K'\|_2  + \|g\|_\infty \|N\|_\infty^{1/2} \|K'\|_1\sqrt{\frac{2L'_h}{nh^2}} + 2 c(\ve)  \|g\|_\infty \|K'\|_2 \frac{L'_h}{n h^{3/2}}.\]
Let
$\theta>0$ and let  \[L'_h=\frac{\theta^2 \|K'\|_2^2}{2\|N\|_\infty\|K'\|_1^2 h}.\]
Obviously the series in \eqref{finalderive} is finite and
we have:
\[M'_h\leq (1+\ve+\theta)\|g\|_\infty\frac{\|K'\|_2}{\sqrt{nh^3}}+\frac{c(\ve) \theta^2 \|K'\|_2^3\|g\|_\infty}{\|N\|_\infty \|K'\|_1^2}\frac{1}{nh^{5/2}}.\]
But $h^2\geq (\tilde\delta n)^{-1}$. Hence
\[M'_h\leq \big(1+\ve+\theta+ \frac{c(\ve) \theta^2 \|K'\|_2^3\|g\|_\infty\sqrt{\tilde\delta}}{\|N\|_\infty \|K'\|_1^2}\big)\frac{\|K'\|_2}{\sqrt{nh^3}}.\]
As previously, it remains to choose $\ve$ and $\theta$ accordingly to conclude.
\end{itemize}
\end{proof}
The second result is based on probabilistic arguments as well.
\begin{lemma}\label{bar}
Under Assumptions and notations of Proposition \ref{rate}, 
if there exists an interval $[a,b]$ in $(0,T)$ such that 
$$[m, M]  :=[\inf_{x\in[a,b]} N(x), \sup_{x\in [a,b]} N(x)] \subset (0,\infty), \qquad Q  :=\sup_{x\in [a,b]} |H(x)|<\infty,$$
and if $\ln(n)\leq D_{\min}\leq n^{1/(2m_0+1)}$ and $n^{1/5}\leq D_{\max} \leq (n/\ln(n))^{1/(4+\eta)}$ for some $\eta>0$ fixed, then  there exists $C_\eta$, a constant depending on $\eta$, such that for $n$ large enough,
$$\PP(\Omega_n^c) \leq C_\eta n^{-q}.$$
\end{lemma}
\begin{proof}
Let $x,y$ be two fixed point of $[a,b]$ and for $h\in\HH'$, first let us look at
$$Z_h(x,y)=\hat{N}_h(x)-\hat{N}_h(y)- [K_h\star N(x)-K_h\star N(y)].$$
One can apply Bernstein inequality to $Z_h(x,y)$ (see {\it e.g.} (2.10) and (2.21) of \cite{mas}) to obtain:
$$\E\big[e^{\lambda Z_h(x,y)}\big] \leq \exp\big[\frac{\lambda^2 v^2(x,y)}{2(1-\lambda c(x,y))}\big], \quad \forall \lambda \in (0, 1/c(x,y)),$$
with
$$c(x,y)=\frac{1}{3n}\sup_{u\in \R}|K_h(x-u)-K_h(y-u)|$$
and
$$v^2(x,y)= \frac{1}{n} \int_{\R} |K_h(x-u)-K_h(y-u)|^2 N(u) du.$$
But, with  $\mathcal{K}$ the Lipschitz constant of $K$,
$$c(x,y)\leq \frac{\mathcal{K}}{3nh^2} |x-y|$$
and 
$$v^2(x,y) \leq \frac{\mathcal{K}^2}{nh^4} |x-y|^2.$$
Let $x_0$ be a fixed point of $[a,b]$. One can apply Theorem 2.1 of \cite{bar} which gives that, if
$$ Z_h=\sup_{x\in[a,b]} |Z_h(x,x_0)|$$
then for all positive $x$,
$$\PP(Z_h\geq 18 (\sqrt{v^2(x+1)}+c(x+1))\leq 2 e^{-x},$$
with 
$$v^2= \frac{\mathcal{K}^2}{nh^4} |b-a|^2$$
and
$$c=\frac{\mathcal{K}}{3nh^2} |b-a|.$$
But one can similarly prove, using Bernstein's inequality, that
$$\PP\big(|\hat{N}_h(x_0)- K_h\star N(x_0)|\geq\sqrt{\frac{\square_{N,K}x}{nh}}+\square_{K}\frac{x}{nh}\big)\leq 2 e^{-x}.$$
Hence 
$$\PP\Big(\exists x\in [a,b], |\hat{N}_h(x)- K_h\star N(x)|\geq\sqrt{\frac{\square_{N,K}x}{nh^4}}+\square_{K}\frac{x}{nh^2}\Big)\leq 4 e^{-x}.$$
We apply this inequality with $x=q \log (n)+D^\eta$, $D=D'_{\min}, \ldots ,D'_{\max}$. We obtain 
{\small
\begin{align*}
& \PP\Big(\exists h \in \HH' ,\exists x\in [a,b], |\hat{N}_h(x)- K_h\star N(x)|\geq\sqrt{\frac{\square_{N,K,q}\log (n)}{nh^{4+\eta}}}+\square_{K,q}\frac{\log (n)}{nh^{2+\eta}}\Big) \\
\leq\; & 4 \sum_{D=1}^{D_{\max}}e^{-D^\eta} n^{-q} =  \square_\eta n^{-q} .
\end{align*}
}
Consequently, since $\frac{\log (n)}{nh^{4+\eta}}\to 0$ uniformly in $h\in\HH'$,  for $n$ large enough,
$$\PP(\exists h \in \HH', \exists x\in [a,b], |\hat{N}_h(x)- K_h\star N(x)|\geq  m/4)\leq \square_\eta n^{-q} .$$
But since $N\in  {\mathcal W}^{s+1}$, $s\geq 1$, then $N'\in  {\mathcal W}^{s}$ so $N'$ is bounded on $\R$. Then
 \begin{eqnarray*}
 K_h\star N(x)-N(x)&=&\int K_h(x-y)(N(y)-N(x))dy\\
 &=&\int K\big(\frac{x-y}{h}\big)\big(\frac{y-x}{h}\big)\int_0^1N'(x+u(y-x))dudy\\
 &=&h\int t K\big(t\big)\int_0^1N'(x-uht)dudt.
 \end{eqnarray*}
 So,
$$ |N(x)-K_h\star N(x)| \leq \square_{N,K} h.$$
Because of the definition of $\HH'$ this term tends to $0$, uniformly in $x$ and $h$. Hence for $n$ large enough, for all $x\in[a,b]$ and all $h\in\HH'$
$$K_h\star N(x)\geq N(x) - m/4 \geq 3m/4.$$
Consequently,  for $n$ large enough,
$$\PP(\exists h \in \HH', \exists x\in [a,b], 2\hat{N}_h(x)<  m)\leq \square_\eta n^{-q} .$$
\end{proof}
\subsubsection*{Rosenthal-type inequality}
The following result studies the behavior of the moments of $\hat\rho_n$.
\begin{lemma}\label{hatrho}
For any $p\geq 2$, 
\[\E\big[|\hat\rho_n-\rho_g(N)|^{p}\big]\leq \square_{p,g,N,c}n^{-p/2}.\]
\end{lemma}
\begin{proof}
 We have:
\[\E\big[|\hat\rho_n-\rho_g(N)|^{p}\big]\leq \square_{p}(A_n+B_n),
\]
with
\[A_n:=\E\Big[\Big|\frac{\sum_{i=1}^nX_i}{n\int g(x)N(x)dx+c}-\frac{\int_{\R_+}xN(x)dx}{\int_{\R_+}g(x)N(x)dx}\Big|^p\Big]\]
and
\[B_n:=\E\Big[\Big|\frac{\sum_{i=1}^nX_i}{\sum_{i=1}^ng(X_i)+c}-\frac{\sum_{i=1}^nX_i}{n\int g(x)N(x)dx+c}\Big|^p\Big].\]
We use the Rosenthal inequality (see {\it e.g.} the textbook \cite{HH}): if $R_1,\ldots,R_n$ are independent centered variables such that
$\E[|R_1|^{p}]<\infty,$  with $p\geq 2$ then 
\begin{eqnarray}\label{rosenthal}
 \E\big[\big|\frac{1}{n}\sum_{i=1}^n R_i\big|^{p}\big] \leq \square_p n^{-p}\big(n\E|R_1|^{p}+\big(n\E R_1^{2}\big)^{p/2}\big).
  \end{eqnarray}
We recall that Assumption \ref{as:an} ensures that  $\E[X_i^p]=\int x^p N(x) dx<\infty$, for any $p\geq 1$.
Hence, 
for the first term $A_n$, using  (\ref{rosenthal}), we have:
\begin{eqnarray*}
 A_n&=&\E\Big[\Big|\frac{\frac{1}{n}\sum_{i=1}^nX_i}{\int g(x)N(x)dx+\frac{c}{n}}-\frac{\int_{\R_+}xN(x)dx}{\int_{\R_+}g(x)N(x)dx}\Big|^p\Big]\\
 &\leq&\square_{g,N,c}\E\Big[\Big|\frac{1}{n}\sum_{i=1}^nX_i\int_{\R_+}g(x)N(x)dx-\int_{\R_+}xN(x)dx\big(\int g(x)N(x)dx+\frac{c}{n}\big)\Big|^p\Big]\\
 &\leq&\square_{p,g,N}\E\Big[\Big|\frac{1}{n}\sum_{i=1}^nX_i-\int_{\R_+}xN(x)dx\Big|^p\Big]+\square_{g,N,c}n^{-p}\\
 &\leq&\square_{p,g,N,c}n^{-p/2}.
 \end{eqnarray*}
Let us turn to the term $B_n$: 
\begin{eqnarray*}
 B_n&=&\E\big[\big|\frac{\sum_{i=1}^nX_i}{\sum_{i=1}^ng(X_i)+c}-\frac{\sum_{i=1}^nX_i}{n\int g(x)N(x)dx+c}\big|^p\big]\\
 &\leq&\big(\E\big[\big|\frac{1}{n}\sum_{i=1}^nX_i\big|^{2p}\big]\big)^{1/2}\times\big(\E\big[\big|\frac{\frac{1}{n}\sum_{i=1}^ng(X_i)-\int g(x)N(x)dx}{\big(\frac{1}{n}\sum_{i=1}^ng(X_i)+\frac{c}{n}\big)\big(\int g(x)N(x)dx+\frac{c}{n}\big)}\big|^{2p}\big]\big)^{1/2}\\
 &\leq &\|gN\|_1^{-p}\big(\E\big[\big|\frac{1}{n}\sum_{i=1}^nX_i\big|^{2p}\big]\big)^{1/2}
 \big(\E\big[\big|\frac{\frac{1}{n}\sum_{i=1}^ng(X_i)-\int g(x)N(x)dx}{\frac{1}{n}\sum_{i=1}^ng(X_i)+\frac{c}{n}}\big|^{2p}\big]\big)^{1/2}\\
 &\leq&\square_{p,g,N}\big(\E\big[\big|\frac{1}{n}\sum_{i=1}^nX_i-\E[X_1]\big|^{2p}+\big|\E[ X_1]\big|^{2p}\big)^{1/2} \times \\
& & \big(\E\big[\big|\frac{\frac{1}{n}\sum_{i=1}^ng(X_i)-\int g(x)N(x)dx}{\frac{1}{n}\sum_{i=1}^ng(X_i)+\frac{c}{n}}\big|^{2p}\big]\big)^{1/2}\\
 &\leq&\square_{p,g,N}\big(\E\big[\big|\frac{\frac{1}{n}\sum_{i=1}^ng(X_i)-\int g(x)N(x)dx}{\frac{1}{n}\sum_{i=1}^ng(X_i)+\frac{c}{n}}\big|^{2p}\big]\big)^{1/2}.
\end{eqnarray*}
Now, we set for $\gamma>3p$
\[\Omega_n=\Big\{\big|\frac{1}{n}\sum_{i=1}^ng(X_i)-\int g(x)N(x)dx\big|\leq\sqrt{\frac{2\gamma\Var(g(X_1))\log n}{n}}+\frac{\gamma\|g\|_\infty\log n}{3n}\Big\}\]
(recall that Assumption~\ref{as:an} states that $\|g\|_\infty<\infty$, which also implies $\E[g(X_1)^2]<\infty$). Since $g$ is positive, the Bernstein inequality (see Section~2.2.3 of \cite{mas}) gives:
\begin{equation}\label{bernstein}
{\mathbb P}(\Omega_n)\geq 1-2n^{-\gamma}.
\end{equation}
Therefore, we bound from above the term $B_n$ by a constant $\square_{p,g,N}$ times 
{\small
\begin{eqnarray*}
& & \big(\E\big[\big|\frac{\frac{1}{n}\sum_{i=1}^ng(X_i)-\int g(x)N(x)dx}{\frac{1}{n}\sum_{i=1}^ng(X_i)+\frac{c}{n}}\big|^{2p}{\bf 1}_{\Omega_n}+\big|\frac{\frac{1}{n}\sum_{i=1}^ng(X_i)-\int g(x)N(x)dx}{\frac{1}{n}\sum_{i=1}^ng(X_i)+\frac{c}{n}}\big|^{2p}{\bf 1}_{\Omega_n^c}\big]\big)^{1/2}\\
 &\leq&\square_{p,g,N,c}n^{-p/2}+ \square_{p,g,N} \sqrt{2}(nc^{-1}\|g\|_\infty)^pn^{-\gamma/2}\\
 &\leq& \square_{p,g,N,c} n^{-p/2},
 \end{eqnarray*}
}
where we have used (\ref{rosenthal}) for the first term and (\ref{bernstein}) for the second one. This concludes the proof of the lemma.
\end{proof}
\subsubsection*{The boundedness of ${\cal L}_k^{-1}$}
\begin{lemma}\label{lk-1}
For any function $\varphi$, we have:
$$\norme{{\mathcal L}^{-1}_k(\varphi)}_{2,T} \leq \sqrt{\frac{1}{3}} \norme{\varphi}_{2,T}.$$
\end{lemma}
\begin{proof}
We have:
$$ \begin{array}{lll} \int |{\mathcal L}^{-1}_k(\varphi)(x)|^2 dx &=&
                                                 \f{T}{k}\sum\limits_{i=0}^{k-1} |H_{i,k} (\varphi)|^2 = \f{T}{k} \sum\limits_{i=0}^{k-1} \biggl(\f{1}{4}(H_{\f{i}{2},k}(\varphi) + \varphi_{\f{i}{2},k})\biggr)^2 \\ \\
&=&\f{T}{k}\biggl( \sum\limits_{j=0}^{[\f{k-1}{2}]}\f{1}{16}\bigl(H_{j,k}  + \varphi_{j,k}\bigr)^2 + \sum\limits_{j=1}^{[\f{k}{2}]} \f{1}{64}(H_{j,k} + \varphi_{j,k} + H_{j-1,k} + \varphi_{j-1,k})^2 \biggr)\\ \\
&\leq&\f{T}{k} \biggl(\sum\limits_{j=0}^{[\f{k-1}{2}]} \f{1}{8}\bigl(H_{j,k}^2 + \varphi_{j,k}^2\bigr) + \sum\limits_{j=1}^{[\f{k}{2}]} \f{1}{16} (H_{j,k}^2 + \varphi_{j,k}^2 + H_{j-1,k}^2 + \varphi_{j-1,k}^2\bigr) \biggr)\\ \\
&\leq & \f{1}{4} \f{T}{k} \sum\limits_{i=0}^{k-1} \biggl(H_{j,k}^2 +\varphi_{j,k}^2\biggr)=\f{1}{4} \biggl( \int |{\mathcal L}^{-1}_k(\varphi)(x)|^2 dx +\f{T}{k}\sum\limits_{i=0}^{k-1}\varphi_{j,k}^2  \biggr). 
\end{array}
$$
At the second line, we have distinguished the $i$'s that are even and the $i$'s that are odd. At the third line, we have used the inequalities $(a+b)^2\leq 2a^2+2b^2$ and $(a+b+c+d)^2\leq 4(a^2+b^2+c^2+d^2).$ By substraction, we obtain
$$ \int |{\mathcal L}^{-1}_k(\varphi)(x)|^2 dx \leq \f{1}{3}\f{T}{k}\sum\limits_{i=0}^{k-1}\varphi_{j,k}^2.
$$
The Cauchy-Schwarz inequality gives:
$$\biggl(\int\limits_{x_{i,k}}^{x_{i+1,k}} \varphi^2(x) dx\biggr)^2 \leq \f{T}{k} \int\limits_{x_{i,k}}^{x_{i+1,k}} \varphi^2(x) dx,$$
so that
$$\f{T}{k}\sum\limits_{i=0}^{k-1}\varphi_{j,k}^2=\f{T}{k}\sum\limits_{i=0}^{k-1}\f{k^2}{T^2} \big(\int\limits_{x_{i,k}}^{x_{i+1,k}} \varphi(x) dx\big)^2\leq \int\limits_0^T \varphi^2(x) dx$$
and finally we obtain the desired result:
$$ \int |{\mathcal L}^{-1}_k(\varphi)(x)|^2 dx  \leq \f{1}{3} \int \varphi^2(x) dx.
$$
\end{proof}

\noindent{\bf Acknowledgment:}  We warmly thank Oleg Lepski for fruitful discussions that eventually led to the choice of his method for the purpose of this article.
The research of M. Hoffmann is supported by the Agence Nationale de la Recherche, Grant No. ANR-08-BLAN-0220-01.
The research of P. Reynaud-Bouret and V. Rivoirard is partly supported by the Agence Nationale de la Recherche, Grant No. ANR-09-BLAN-0128 PARCIMONIE.
The research of M. Doumic is partly supported by the Agence Nationale de la Recherche, Grant No. ANR-09-BLAN-0218 TOPPAZ.

\end{document}